\documentclass[moor,sglanonrev]{informs4}
\usepackage{eqndefns-left} % For checking the display equation width and equation environment definitions %
\RequirePackage{tgtermes}
\RequirePackage{newtxtext}
\RequirePackage{newtxmath}
\RequirePackage{bm}
\RequirePackage{endnotes}

\OneAndAHalfSpacedXII % Current default line spacing
\usepackage[ruled]{algorithm2e}
\usepackage{hyperref}
\usepackage{tikz}
\usepackage{caption}
\usepackage{subcaption}
\usepackage{float}
\usepackage[sort&compress]{natbib}
 \bibpunct[, ]{[}{]}{,}{n}{}{,}%
 \def\bibfont{\small}%
\EquationsNumberedThrough    % Default: (1), (2), ...
\TheoremsNumberedThrough     % Preferred (Theorem 1, Lemma 1, Theorem 2)
\ECRepeatTheorems  %  

\MANUSCRIPTNO{MOOR-0001-2024.00}

\begin{document}
%%%%%%%%%%%%%%%%

%\setcounter{page}{1}
%\DOI{10.1287/xxxx.0000.0000}%

\RUNAUTHOR{Di Zhang, Yihang Zhang and Suvrajeet Sen}

% \RUNTITLE{Predictive Maintenance in Manufacturing}
\RUNTITLE{Adaptive Sampling-based PHA for SP}

% \TITLE{Optimal Resource Allocation in Humanitarian Logistics: A Stochastic Programming Approach}
\TITLE{An Adaptive Sampling-based Progressive Hedging Algorithm for Stochastic Programming}

\ARTICLEAUTHORS{%
%\AUTHOR{John Doe,\textsuperscript{a} Jane Smith,\textsuperscript{b}}
%\AFF{\textsuperscript{a}Department of Industrial Engineering, University of XYZ, \EMAIL{john.doe@xyz.edu; \textsuperscript{b}Department of Computer Science, University of ABC, \EMAIL{jane.smith@abc.edu}} 
\AUTHOR{Di Zhang}
\AFF{Department of Industrial and Systems Engineering, University of Southern California \EMAIL{dzhang22@usc.edu}}

\AUTHOR{Yihang Zhang}
\AFF{Department of Industrial and Systems Engineering, University of Southern California \EMAIL{zhangyih@usc.edu}}

\AUTHOR{Suvrajeet Sen}
\AFF{Department of Industrial and Systems Engineering, University of Southern California \EMAIL{suvrajes@usc.edu}}
} % end of the block

\ABSTRACT{%
% Enter your abstract
The Progressive Hedging algorithm is a cornerstone in the realm of large-scale stochastic programming. However, its traditional implementation faces several limitations, including the requirement to solve all scenario subproblems in each iteration, reliance on an explicit probability distribution, and a convergence process that is highly sensitive to the choice of proximal parameter. This paper introduces a sampling-based PH algorithm aimed at overcoming these challenges. Our approach employs a dynamic sampling process for the number of scenario subproblems solved in any iteration, incorporates adaptive sequential sampling to determine sample sizes, utilizes a stochastic conjugate subgradient method for direction finding, while applying a line-search technique to update the dual variables. Experimental results demonstrate that this novel algorithm not only addresses the bottlenecks mentioned above but also potentially overcomes its scalability for large data sets.
}%

\FUNDING{This research was supported by grant number DE-SC0023361, sub-award G002122-7510 from the Dept. of Energy}

%Supplemental Material:
%Data Ethics & Reproducibility Note:

%\KEYWORDS{Stochastic programming, Decision support,Uncertainty, Disaster response, Optimization}

%\HISTORY{Received: Month DD, YYYY; Accepted: Month DD, YYYY; Published Online: Month DD, YYYY}

\maketitle
%%%%%%%%%%%%%%%%%%%%%%%%%%%%%%%%%%%%%%%%%%%%%%%%%%%%%%%%%%%%%%%%%%%%%%
% Text of your paper here
\section{Introduction}
\subsection{Problem setup}
In this paper, we focus on solving stochastic programming (SP) problems stated as follows.
\begin{equation} \label{SP}
\min_{x \in X \subseteq \mathbb{R}^n} f(x) = E_{\tilde \omega}[h(x,\tilde \omega)],
\end{equation}

\noindent where $X$ is a non-empty convex set, $\tilde\omega \in \mathbb{R}^m$ denotes a random vector, and $h(\cdot,\tilde \omega)$ is a finite-valued convex function in $x$ for any particular realization $\tilde \omega = \omega$. Except in very simplistic situations (e.g., when $\tilde \omega$ has a small number of outcomes), solving the optimization problem  \eqref{SP} using deterministic algorithms may be infeasible. To overcome challenges arising from efforts to solve realistic instances, the focus often shifts towards solving a sample-average approximation (SAA) problem. In such an approach, we create an approximation by using $|S|$ i.i.d. realizations of $\tilde \omega$ called scenarios, and denote them by $\omega^1,...,\omega^{|S|}$. In this setting, \eqref{SP} can be reformulated as 

\begin{equation} \label{SP2}
    \min_{x \in X} f(x) = \frac{1}{|S|} \sum_{s=1}^{|S|} h(x, \omega^s).
\end{equation}

\noindent In order to obtain an optimal solution of \eqref{SP2}, many algorithms have been proposed, and one of the more popular algorithms among them is the Progressive Hedging (PH) algorithm~\cite{RW1991}.

The goal of this paper is to present a PH-based algorithm that enables adaptive sampling, allowing PH to accommodate both continuous and discrete random variables with equal ease, while incorporating many of the strengths of the original PH algorithm. To lay the groundwork, we will commence by reviewing the classic PH algorithm.  Following that discussion, we will briefly summarize some of the key elements which constitute the major innovations necessary for all pieces of the new algorithm to dovetail with each other.  Once this overview is completed, we will embark on the new mathematical constructs which constitute the main contributions of this paper.

 \subsection{Classic progressive hedging algorithm}
The PH Algorithm is a predecessor of the alternating direction method of multipliers (ADMM)~\cite{B2011} which has attracted significant attention in the optimization community, especially in the SP community, and extensively used for addressing scenario-based uncertainty in large-scale optimization. Originally conceptualized by Rockafellar and Wets in 1991~\cite{RW1991}, the algorithm has evolved significantly, finding applications in diverse fields such as finance, energy, logistics~\cite{WW2011} and others. The convergence properties of the PH algorithm have been rigorously studied since its inception. Rockafellar and Wets provided the initial proof of convergence for convex problems, and more recently for non-convex and other generalizations (e.g., Ruszczyński and Shapiro~\cite{RS2003}).  However, despite its widespread popularity, the classic PH algorithm remains a deterministic method which immediately limits its applicability for broader stochastic optimization problems, especially those for which sampling may be critical. Such instances arise in applications where the uncertainty can only be sampled via a simulation or the sample space is so vast that it is unrealistic to expect to have knowledge of all possible outcomes (and probability distribution) associated with the data.  More specifically, consider the following limitations. 

\begin{itemize}
    \item The convergence criteria of the PH algorithm are predicated upon the accessibility and knowledge of the probability distribution of the random variables involved. This requirement severely restricts the algorithm's utility in situations where such a distribution remains unknown or can only be simulated, thereby restricting its range of applications.
    
    \item  Second, in contrast to several SP algorithms, notably stochastic gradient descent (SGD), the classic PH algorithm requires the solution of all subproblems during each iteration. This requirement constitutes a significant computational bottleneck, and substantially impedes its scalability for cases in which the number of scenarios are far too many to enumerate in each iteration. 

    \item  Lastly, the algorithm employs a dual update rule that is fundamentally gradient-based and may only be effective locally (i.e., at the current dual iterate). This approach makes the convergence process highly sensitive to the selection of the proximal parameter, which functions as the step size in the dual update. Consequently, the algorithm's performance depends heavily on the choice of the proximal parameter, further complicating its practical implementation.  The reader may consult Naepels and Woodruff \cite{naepels_gradient-based_2023} which provides a history of heuristics which have been investigated over the past decade.
    
    For instances with a large number of scenarios (or even continuous random variables), an adaptive sampling algorithm may be more suitable because such methods allow for the possibility of creating approximations which are able to sequentially adapt to the data observed during the course of the algorithm. In this paper, we will introduce a sampling-based PH algorithm which {\it adaptively} samples a subset of scenario subproblems in each iteration, facilitating a dynamic selection of scenarios, as well as updates to enhance computational realism for truly large-scale optimization (with far too many outcomes to enumerate in each iteration). 

% Moreover, distinct from traditional gradient-based dual update, our algorithm employs a stochastic conjugate subgradient direction. This method has been demonstrated to yield favorable convergence outcomes for non-smooth convex stochastic programming~\cite{W1974, W1975}. Additionally, to optimize the dual variable updates, we propose a departure from fixed step sizes. Specifically, we incorporate a line-search technique alongside a trust region framework, aiming to determine the optimal step size. This adjustment is poised to significantly refine the convergence process, offering a robust resolution of the limitations of conventional PH algorithm mentioned above.
\end{itemize}
\subsection{Adaptive sampling}
When the probability distribution is not available, the PH algorithm typically employs sample average approximation (SAA)~\cite{S1998} which has a fixed sample size to estimate the uncertainty. However, it is difficult to determine the sample size~\cite{S2016,DS2024} when applying such methods to the PH algorithm and convergence properties have been established only for cases
with a finite number of outcomes. While  adaptive sampling~\cite{M1999} has emerged as a powerful approach for instances with a large number of outcomes (especially for multi-dimensional uncertainty, see~\cite{HS1991, HS1994, peng2024asymmetric,lin2024economic}), such adaptive sampling has remained elusive in the context of PH. Adaptive sampling dynamically adjusts the sample size based on the ongoing optimization progress, and has also been designed for variance reduction by introducing the notion of compromise decisions \cite{sen2016}. 

It so happens that there have been some efforts devoted to combining adaptive sampling with the PH algorithm. For instance, \cite{B2020} provides randomized PH methods which is able to produce new iterates as soon as a single scenario subproblem is solved. \cite{A2012} proposes a SBPHA approach which is a combination of replication and PH algorithm. While these methods provide ``randomized" PH algorithms, they are restricted to sampling from a \textit{fixed, finite collection} of scenarios, without adaptively determining the sample size based on the optimization instance. On the other hand, our sampling-based PH algorithm will strategically sample a subset of scenario subproblems in each iteration, facilitating dynamic scenario updates to enhance computational efficiency. We will also leverage sample complexity analysis~\cite{V1998} to support the choice of the ultimate sample size used at termination. This approach ensures a theoretically grounded, and computationally realistic approach.

\subsection{Stochastic conjugate subgradient method}

The dual update of the classic PH algorithm can be considered as an iteration with a constant step size in the gradient direction. The contraction property induced by the regularization can be used to prove convergence of the deterministic PH algorithm. % deterministic first-order method~\cite{RM1951, B2007}. 
While such a gradient-based dual update can also be adopted in the sampled setting as in \cite{B2020}, it is unclear whether convergence of the overall sampling-based PH algorithm can be proved in the case of infinitely many scenarios (or continuous random variables).  

On the other hand, one might consider using the reasonably obvious strategy of updating the dual multipliers of PH by using a stochastic first-order (SFO) method.  However, such a modification would not yield the perfect analogy of the deterministic PHA in the sense that we wish to preserve a fundamental property of PHA, namely, its contraction property. By combining the PHA with a new stochastic conjugate subgradient algorithm \cite{DS2024}, we will devise a complete analog of the deterministic PHA in the sense that it will also ensure a stochastic contraction property for the new sampling-based PHA.  

%how one might extend the deterministic PH algorithm to a first-order gradient-based method.  which have dominated the field  several convincing reasons: low computational cost, simplicity of implementation, and strong empirical results. However, SFO methods are not known for their numerical accuracy, especially when the functions are non-smooth, and the classic convergence rate~\cite{NB2001} only applies under the assumption that the component functions are smooth.

In contrast to first-order methods, algorithms based on conjugate (sub)gradients (CG) provide finite termination when the objective function is a linear-quadratic deterministic function~\cite{N2006}. Furthermore, CG algorithms are also used for large-scale non-smooth convex optimization problems~\cite{W1975, Y2016}. However, despite their many strengths (e.g., fast per-iteration convergence, frequent explicit regularization of step size, and better parallelization than first-order methods), the use of conjugate subgradient methods for SP is uncommon. Although~\cite{J2018, Y2022} proposed some new stochastic conjugate gradient (SCG) algorithms, their assumptions do not support non-smooth objective functions, and are, therefore, not directly applicable to our class of problems.

To solve stochastic optimization problems where functions are non-smooth (as in \eqref{SP2}), one might consider using a stochastic subgradient method. In this paper, we treat the dual optimization problem in PH as a stochastic linear-quadratic optimization problem and propose leveraging a stochastic conjugate subgradient (SCS) direction, originating from~\cite{W1974, W1975, ZS2024, ZS, zhang2024stochastic}, which accommodates both the curvature of a quadratic function and the decomposition of subgradients by data points, as is common in stochastic programming~\cite{HS1991, HS1994}. This combination enhances the power of stochastic first-order approaches without the additional burden of second-order approximations. In other words, our method shares the spirit of Hessian-Free (HF) optimization, which has gained attention in the machine learning community~\cite{MJ2010, CScheinberg2017,zhang2024general}. As our theoretical analysis and computational results reveal, the new method provides a ``sweet spot" at the intersection of speed, accuracy, and scalability of algorithms (see Theorems \ref{Convergence rate} and \ref{optimality}).

\subsection{Penalty parameter and line-search algorithm}

In the original Progressive Hedging (PH) algorithm, Rockafellar and Wets briefly discussed the effect of the penalty parameter on the performance of the PH algorithm and the quality of the solution, but they did not provide any guidance on how to choose it. Since then, various choices have been proposed, often differing depending on the application. Readers interested in the selection of the penalty parameter can refer to~\cite{Z2016}, which offers a comprehensive literature review of related methods.

Our approach differs from these methods by employing a stochastic direction-finding technique~\cite{PS2020, B2014} derived from Wolfe's line-search method~\cite{W1974, W1975}. Compared with existing approaches, our method is more mathematically motivated and adheres closely to optimization principles~\cite{B2004}.

\subsection{Convergence Analysis using Modern Stochastic Programming Framework}

The classic PH algorithm utilizes the deterministic fixed-point theorem to analyze its convergence properties~\cite{RW1991}. However, this approach cannot be directly applied to our method due to fundamental differences between the two algorithms. While the classic PH algorithm generates a fixed number of scenarios at the beginning, our method dynamically adds new scenarios throughout the iterations. As a result, even for scenarios with the same realization $\omega$, their corresponding dual variables can differ significantly in our approach. This discrepancy prevents the direct application of the deterministic fixed-point theorem.

To address this challenge, we adopt a modern stochastic programming framework to establish the convergence rate~\cite{PS2020, B2014}. Specifically, we define a sequence of probability spaces corresponding to a stochastic process. With these well-defined probability spaces and the associated stochastic process, we demonstrate the submartingale property of the dual objective value sequence. By leveraging this property, we apply the optional stopping theorem and the renewal reward theorem~\cite{GGD2020} to rigorously establish the convergence rate of the proposed algorithm.

\section{Adaptive Sampling-based Progressive Hedging Algorithm}

In this section, we focus on solving (\ref{SP}) using a sampling-based PHA in which we combine classic PHA~\cite{RW1991} with adaptive sampling in stochastic programming \cite{BL2011}. The main ingredients of the algorithm include six important components:
\begin{itemize}
\item Sequential scenario sampling. At iteration $k$, different from classic PH Algorithm which solves all scenario subproblems, we will sample $|S_k|$ scenarios and only solve the subproblems corresponding to these scenarios. This is similar to using sample average approximation (SAA) to approximate function $f$ using $f_k$,  
\begin{equation} \label{OSAA}
    f_k(x)= \frac{1}{|S_k|} \sum_{s=1}^{|S_k|} h(x, \omega^s),
\end{equation}
\noindent where $|S_k|$ will be determined based on concentration inequalities in Theorem \ref{concentration inequality}.  

\item Update the primal variables. For each scenario $s \in S_k$, similar to the classic PHA, we first solve the corresponding augmented Lagrangian problem 

\begin{equation} \label{x_k^s}
    x_{k}^s = \argmin_{x^s \in X} h(x^s, \omega^s) + \langle \lambda_k^s, x^s - \bar{x}_{k-1} \rangle + \frac{\rho}{2} ||x^s - \bar{x}_{k-1}||^2.
\end{equation}

Then a feasible solution is 
\begin{equation} \label{x_k}
    \bar{x}_k = \frac{1}{|S_k|} \sum_{s=1}^{|S_k|}  x_k^s.
\end{equation}

% \item Build piece-wise linear function with respect to $\lambda^s$ to approximate the dual objective. Specifically, in iteration $k$ and scenarios $s$, a new linear piece $a_k^s \lambda^s + b_k^s$ will be added with 

% \begin{equation}
%     \begin{aligned}
%         & a_k^s = x_k^s - x_k,\\
%         & b_k^s = h(x_k^s, \omega^s) + \frac{\rho}{2} ||x_k^s - x_k||^2.
%     \end{aligned} 
% \end{equation}

% \noindent The piece-wise linear approximation will be 

% \begin{equation} \label{piece-wise linear function}
%     \hat{g}_k^s(\lambda^s) = \min_{i=1,2,...,k} \{a_i^s \lambda^s + b_i^s\}.
% \end{equation}

\item Conjugate subgradient direction finding. The idea here is inspired by Wolfe's non-smooth conjugate subgradient method, which uses the smallest norm of the convex combination of the previous search direction $d_{k-1}^s$ and the current subgradient 

\begin{equation} \label{g_k^s}
    g_k^s = x_k^s - \bar{x}_k.
\end{equation}
More specifically, we first solve the following one-dimensional QP

\begin{equation} \label{gamma_k^s}
\gamma_k^s = \argmin_{\gamma^s \in [0,1]}\frac{1}{2}||\gamma^s  d_{k-1}^s + (1-\gamma^s) g_k^s||^2.
\end{equation}
The new direction can be defined as 
\begin{equation} \label{d_k^s}
    d_k^s=\gamma_k^s d_{k-1}^s +(1-\gamma_k^s)g_k^s := Nr(G_k^s),
\end{equation}

\noindent where $G_k^s=\{g_k^s,d_{k-1}^s\}$ and $Nr$ is an operator that selects the vector from the set with the smallest 2-norm.

\item Update the dual variables using line-search. We update $\hat{\lambda}_{k+1}^s = \lambda_k^s + \theta_k^s d_k^s$. Define 
\begin{equation} \label{line-search}
        L_k^s(\lambda_k^s + \theta^s d_k^s) :=  h(\hat{x}_k^s, \omega^s) + \langle \lambda_{k}^s+\theta^s d_k^s, \hat{x}_k^s - \hat{x}_{k}\rangle  + \frac{\rho}{2} ||\hat{x}_k^s - \hat{x}_{k}||^2,
\end{equation}

\noindent where $\delta_k^s$ is the stochastic searching region, $\hat{x}_k^s$ and $\hat{x}_{k}$ are obtained by replacing $\lambda_{k}^s$ in \eqref{x_k^s} with $\lambda_k^s + \theta d_k^s$ in each scenario and resolving \eqref{x_k^s} and \eqref{x_k}, respectively. The step size $\theta_k^s $ will be obtained by performing an inexact line-search which makes it satisfy the strong-Wolfe condition~\cite{W1974,W1975}. Let $g^s(\theta^s) \in \partial_{\lambda^s} L_{k}^s(\lambda_{k}^s + \theta^s \cdot d_k^s)$ and define the intervals $L$ and $R$.
\begin{equation} \label{strong-wolfe condition}
    \begin{aligned}
        & L=\{\theta^s >0 \ | \ L_{k}^s(\lambda_k + \theta^s \cdot d_k)-L_{k}^s(\lambda_k) \geq m_1 ||d_k^s||^2 \theta^s \},\\
        & R =\{\theta^s >0 \ | \  \sum_{s=1}^{S_k} \langle g^s(\theta^s),d_k^s \rangle \leq m_2  ||d_k^s||^2 \},
    \end{aligned}
\end{equation}
\noindent where $m_2 < m_1 < 1/2$.  The output of the step size should satisfy two conditions: (i) $L$ includes the step sizes which sufficiently increase the dual objective function approximation $L_{k}^s$, and (ii) $R$ includes the step sizes for which the directional derivative estimate for $L_{k}^s$ is sufficiently decreased. The algorithm seeks points that belong to $L \cap R$. The combination of $L$ and $R$ is called strong-Wolfe condition and it has been shown in \cite{W1975} (Lemma 1) that $L \cap R$ is not empty. Intuitively, the above procedure is a line-search along conjugate direction $d_k^s$ and the goal is to increase the dual objective function $L_k^s$.

\item Stochastic search region update. Similar to the standard stochastic trust region method~\cite{B2014,BC2016}, let $\gamma>1$, $\eta \in (0,1)$, $\lambda_k := (\lambda_k^1,...,\lambda_k^{|S_k|})$ and $L_k(\lambda_k):= \sum_{s=1}^{|S_k|} L_k^s(\lambda_k^s)$, if 

\begin{equation} \label{trust region rule}
    L_k(\hat{\lambda}_{k})-L_k(\lambda_{k-1}) > \eta [L_{k-1}(\hat{\lambda}_{k})-L_{k-1}(\lambda_{k-1})],
\end{equation}
set $\delta_k = \gamma \delta_{k-1}$ and $\lambda_k = \hat{\lambda}_k$. In this case we say iteration $k$ of the algorithm is \textit{successful}. Otherwise, set $\delta_k =  \delta_{k-1}/\gamma$ and $\lambda_k = \lambda_{k-1}$.

\item Termination criteria. Define $d_k := (d_k^1,...,d_k^{|S_k|})$ and $||d_k|| := \frac{1}{|S_k|} \sum_{s=1}^{|S_k|} ||d_k^s||$. The algorithm concludes its process when $||d_k||<\varepsilon$ and $\delta_k = \delta_{min}$. As we will show in Theorem \ref{optimality}, a diminutive value of $||d_k||$ indicates a small norm of $g_k$, fulfilling the optimality condition for an unconstrained convex optimization problem. Additionally, a threshold $\delta_{min}$ is established to prevent premature termination in the early iterations. Without this threshold, the algorithm may halt prematurely even if $||d_k|| < \varepsilon$. However, if $\delta_{min}$ is introduced, based on Theorem \ref{concentration inequality}, $|S_k|$ should be chosen such that 
\begin{equation*}
    |S_k| = -8\log(\varepsilon/2)\cdot \frac{M_1^2}{\kappa^2\delta_k^4},
\end{equation*}   

\noindent where $\kappa$ is the $\kappa$-approximation defined in Definition \ref{kappa approximation}. This effectively mitigates any early stopping concerns with high probability.
\end{itemize}

\begin{algorithm}[h]
\small
\caption{Adaptive Sampling-based Progressive Hedging Algorithm}
\label{SPHA}
Initialize $ \bar{x}_0 \in W$, $\lambda_0 = (\lambda_0^1,...,\lambda_0^{|S_0|}) = \vec{0} $, $d_0 = (d_0^1,d_0^2,...,d_0^{|S_0|})$, $\rho > 0$, $\gamma > 1$, $\eta \in (0,1)$, $\delta_0 \in  [\delta_{min}, \delta_{max}]$, $\varepsilon > 0$ and $k = 1$.

\While{$||d_k|| > \varepsilon $}
{   

    Sample new scenarios $\omega^{|S_{k-1}|+1} \dots \omega^{|S_k|}$. Update scenario set $S_k := S_{k-1} \cup \{ \omega^{|S_{k-1}|+1} \dots \omega^{|S_k|} \}$

    $\lambda_{k-1}^s = \vec{0} \quad \forall s = |S_{k-1}|+1, \dots, |S_k|$
    
    $x_k^s = \argmin_{x^s \in X} h(x^s, \omega^s) + \langle \lambda_{k-1}^s , x^s - \bar{x}_{k-1} \rangle + \frac{\rho}{2} ||x^s - \bar{x}_{k-1}||^2$, $\forall s $.

    $\bar{x}_k = \frac{1}{|S_k|} \sum_{s=1}^{|S_k|} x_{k}^s$.

    Calculate $g_k^s = x_k^s - \bar{x}_k$, $G_k^s= \{ d_{k-1}^s, g_k^s\}$ and $d_k^s = Nr(G_k^s)$, $\forall s $.

    Apply line-search algorithm to find a step size $\theta_k^s$ which belongs to $L \cap R$.

    Update $\hat{\lambda}_k^{s} = \lambda_{k-1}^{s} + \theta_k^s d_k^s$, $\forall s $.

    \eIf{\eqref{trust region rule} is satisfied}
    { $\lambda_k  \leftarrow \hat{\lambda}_k, \quad \delta_k \leftarrow \min \{\gamma\delta_{k-1},\delta_{max} \}$\;
    }{
            $\lambda_k  \leftarrow \lambda_{k-1}, \quad \delta_k \leftarrow \max\{\frac{\delta_{k-1}}{\gamma}, \delta_{min}\}$\;
        }

        $k \leftarrow k+1$
    }
\end{algorithm}

\section{Convergence Properties}

\subsection{Scenario Dual Objective Functions}
\begin{definition}
For any $\omega^s$, $\bar{x}$ and $\lambda^s$, define the scenario optimal solution 

\begin{equation*}
    x^s(\lambda^s, \omega^s) = \argmin_{x \in X} h(x,\omega^s) + \langle \lambda^s, x - \bar{x} \rangle,
\end{equation*}
and the corresponding scenario dual objective function 

\begin{equation*}
    L^s(\lambda^s ; \omega^s, \bar{x}) = h(x^s,\omega^s) + \langle \lambda^s, x^s - \bar{x} \rangle,
\end{equation*}

\end{definition}

\begin{lemma} \label{dual constraint}
    In any iteration of Algorithm \ref{SPHA}, we have $\sum_{s=1}^{|S_k|} \lambda_k^s = 0$.
\end{lemma}

\begin{proof}{Proof}
    We will show the lemma by induction. When $k = 0$, $\lambda_0^s = 0$, for all $s \in S_0$. Thus, the statement is correct. Next, when $k = 1$, we have
    \begin{equation*}
        \lambda_1^s = \rho (x_1^s - \bar{x}_k) = \rho (x_1^s - \frac{1}{|S_1|}\sum_{s = 1}^{|S_1|} x_1^s), 
    \end{equation*}
    \noindent then 
    \begin{equation*}
        \sum_{s=1}^{|S_1|} \lambda_1^s = \rho \sum_{s=1}^{|S_1|}(x_1^s - \frac{1}{|S_1|}\sum_{s = 1}^{|S_1|} x_1^s) = 0.
    \end{equation*}

    \noindent Suppose it is true for $k = n$, i.e., $\sum_{s=1}^{|S_n|} \lambda_n^s = 0$, then for $k = n + 1$, 
    \begin{equation*}
    \sum_{s=1}^{|S_{n+1}|} \lambda_{n+1}^s = \sum_{s=1}^{|S_n|}\lambda_n^s + \sum_{s=|S_n|+1}^{|S_{n+1}|}\lambda_n^s + \rho \sum_{s=1}^{|S_{n+1}|}(x_{n+1}^s - \bar{x}_{n+1}) = 0,
    \end{equation*}
     where the first part is $0$ by the induction hypothesis. The second and third part is $0$ by the construction of Algorithm \ref{SPHA}. We thus conclude the proof. \Halmos
    
\end{proof}

\begin{theorem} \label{independent}
    In every iteration of Algorithm \ref{SPHA}, we have 
    \begin{equation*}
        \frac{1}{|S_k|} \sum_{s = 1}^{|S_k|} L_k^s(\lambda_k^s; \omega^s, \bar{x}_k)  = \frac{1}{|S_k|} \sum_{s = 1}^{|S_k|} h(x_k^s,\omega^s) + \langle \lambda_k^s, x_k^s \rangle.
    \end{equation*}
\end{theorem}

\begin{proof}{Proof}
    \begin{equation*}
        \frac{1}{|S_k|} \sum_{s = 1}^{|S_k|} L_k^s(\lambda^s; \omega^s, \bar{x}_k) = \frac{1}{|S_k|} \sum_{s = 1}^{|S_k|} [h(x_k^s,\omega^s) + \langle \lambda_k^s, x_k^s \rangle] -  \frac{1}{|S_k|} \sum_{s = 1}^{|S_k|} \langle \lambda_k^s, \bar{x}_k \rangle,
    \end{equation*}
    \noindent where the last part is $0$ by Lemma \ref{dual constraint}. \Halmos
\end{proof}

\textbf{Remark:} Theorem \ref{independent} indicates that the dual objective function value is independent to $\bar{x}$. We can thus define the scenario dual objective function as 

\begin{equation*}
    L^s(\lambda^s ; \omega^s) := h(x^s,\omega^s) + \langle \lambda^s, x^s \rangle,
\end{equation*}
which will not affect the calculation of the dual objective function value.

\subsection{Stochastic Local Approximation}

\begin{assumption} \label{boundedness of primal variable}
    The primal feasible region $X$ is bounded, i.e., there exists a constant $C$ such that for any $x \in X$, we have $||x|| \leq C$.
\end{assumption}
    
\begin{assumption} \label{boundedness of dual objective}
    For any $\omega^s$ and $\lambda^s$, $L^s(\lambda^s; \omega^s)$ is bounded, i.e., there exists a constant $M$, such that $|L^s(\lambda^s; \omega^s)| \leq M$.
\end{assumption}

\begin{assumption} \label{Lipschitz continuity}
    For any $\omega^s$, $L^s(\lambda^s; \omega^s)$ possesses a Lipschitz constant $0 < L < \infty$ {\color{blue}(in $\lambda$)} and the sub-differentials are non-empty for every $\lambda^s$.
\end{assumption}

\begin{assumption} \label{Boundedness of subgradient}
    For any $\omega^s$ and $\lambda^s$, $L^s(\lambda^s; \omega^s)$ has bounded sub-differential set, i.e., for any $g^s \in \partial_{\lambda^s}L^s(\lambda^s; \omega^s)$, we have $||g^s|| \leq G$.
\end{assumption}

\begin{definition} \label{kappa approximation}  A multi-variable function $L_k(\lambda^1,...,\lambda^{|S|})$ is a local $\kappa$-approximation of $L(\lambda^1,...,\lambda^{|S|})$ on $\mathcal{B}(\lambda_k^s,\delta_k)$, if for all $\lambda^s \in \mathcal{B}(\lambda_k^s,\delta_k)$,
\begin{equation*}
|L(\lambda^1,...,\lambda^{|S|})-L_k(\lambda^1,...,\lambda^{|S|})| \leq \kappa \delta_k^2.
\end{equation*}
\end{definition}

\begin{definition} \label{stocahstic kappa approximation}Let $\Delta_k$, $\Lambda_k^s$ denote the random counterpart of $\delta_k$,  $\lambda_k^s$, respectively, and $\mathbb{L}_{k}$ denote the $\sigma$-algebra generated by $ \{\mathcal{L}_{i}\}_{i=0}^{k-1}$.  A sequence of random model $\{\mathcal{L}_{k}\}$ is said to be $\mu_1$-probabilistically $\kappa$-approximation of $\mathcal{L}$ with respect to $\{\mathcal{B}(\Lambda_k^s,\Delta_k)\}$ if the event 
\begin{equation*}
    I_k \overset{\Delta}{=}  \{ \mathcal{L}_{k} \ \textit{is} \ \kappa- \textit{approximation of}  \ \mathcal{L} \ \textit{on} \ \mathcal{B}(\Lambda_k^s,\Delta_k)\}
\end{equation*}

\noindent satisfies the condition 
\begin{equation*}
    \mathbb{P}(I_k|\mathcal{L}_{k-1}) \geq \mu_1 > 0.
\end{equation*}
\end{definition}

\begin{lemma} \label{boundedness lemma 1}
    If a scenario $ \omega^{s_0} \in S_k \backslash S_{k-1}$ added during iteration $k$ is replaced by $(\omega^{s_0})'$, then for any $s \neq s_0$, we have 
    \begin{equation*}
        ||g_k^s - (g_k^s)'|| = O(\frac{1}{|S_k|}), \  |\gamma_k^s - (\gamma_k^s)'| = O(\frac{1}{|S_k|}), \ \textit{and} \  ||d_k^s - (d_k^s)'|| = O(\frac{1}{|S_k|}),
    \end{equation*}
    where \( g_k^s \), \( \gamma_k^s \), and \( d_k^s \) are defined in equation \eqref{g_k^s}-\eqref{d_k^s}. \( (g_k^s)' \), \( (\gamma_k^s)' \), and \( (d_k^s)' \) are the iterates obtained through Algorithm 1 when the newly added sample \( \omega^{s_0} \) is replaced by \( (\omega^{s_0})' \).
\end{lemma}

\begin{proof}{Proof}
    First, note that  for $s \neq s_0$,
    \begin{equation*}
        \begin{aligned}
            & g_k^s = x_k^s - \bar{x}_k = x_k^s - \frac{x_k^1 + ... + x_k^s + ... + x_k^{|S_k|}} {|S_k|}, \\
             & (g_k^{s})' = x_k^s - \bar{x}_k' = x_k^s - \frac{x_k^1 + ... + (x_k^{s})' + ... + x_k^{|S_k|}} {|S_k|}.
        \end{aligned} 
    \end{equation*}
    \noindent Thus, by Assumption \ref{boundedness of primal variable}, 
    \begin{equation} \label{boundedness of g}
        ||g_k^s - (g_k^{s})'|| = \frac{||x_k^s-(x_k^s)'||}{|S_k|} \leq \frac{2C}{|S_k|}.
    \end{equation}

    \noindent Also, note that 

    \begin{equation*}
        \begin{aligned}
            & \gamma_k^s = \argmin_{\gamma^s \in [0,1]} ||\gamma^s  d_{k-1}^s + (1-\gamma^s) g_k^s||^2 = \frac{\langle d_{k-1}^s, d_{k-1}^s - g_k^s \rangle}{||d_{k-1}^s - g_k^s||^2}, \\
            & (\gamma_k^s)' = \argmin_{\gamma^s \in [0,1]} ||\gamma^s  d_{k-1}^s + (1-\gamma^s) (g_k^s)'||^2 = \frac{\langle d_{k-1}^s, d_{k-1}^s - (g_k^s)' \rangle}{||d_{k-1}^s - (g_k^s)'||^2}.     
        \end{aligned}   
    \end{equation*}
    
    \noindent Let $\Delta g_k^s =(g_k^s)' - g_k^s$, we have 

    \begin{equation*}
        \begin{aligned}
            \gamma_k^s - (\gamma_k^s)' & = \frac{\langle d_{k-1}^s, d_{k-1}^s - g_k^s \rangle \cdot ||d_{k-1}^s - g_k^s - \Delta g_k^s||^2}{||d_{k-1}^s - g_k^s||^2 \cdot ||d_{k-1}^s - g_k^s - \Delta g_k^s||^2} \\
            & - \frac{\langle d_{k-1}^s, d_{k-1}^s - g_k^s - \Delta g_k^s \rangle \cdot ||d_{k-1}^s - g_k^s||^2}{||d_{k-1}^s - g_k^s||^2 \cdot ||d_{k-1}^s - g_k^s - \Delta g_k^s||^2} \\
            & = \frac{\langle \Delta g_k^s, d_{k-1}^s - g_k^s  \rangle \cdot \langle d_{k-1}^s, d_{k-1}^s - g_k^s \rangle} {||d_{k-1}^s - g_k^s||^2 \cdot ||d_{k-1}^s - g_k^s - \Delta g_k^s||^2} \\
            & + \frac{||\Delta g_k^s||^2 \cdot \langle d_{k-1}^s, d_{k-1}^s - g_k^s \rangle - \langle \Delta g_k^s, d_{k-1}^s \rangle \cdot ||d_{k-1}^s - g_k^s||^2}{||d_{k-1}^s - g_k^s||^2 \cdot ||d_{k-1}^s - g_k^s - \Delta g_k^s||^2}.
        \end{aligned}
    \end{equation*}

    \noindent With Assumption \ref{Boundedness of subgradient}, Eq. \eqref{boundedness of g} and the triangle inequality, we have 

    \begin{equation} \label{boundedness of gamma}
        |\gamma_k^s - (\gamma_k^s)'| = O(\frac{1}{|S_k|})
    \end{equation}

    \noindent Finally, to bound $||d_k^s - (d_k^s)'||$, note that

    \begin{equation}
        \begin{aligned}
            & d_k^s = \gamma_k^s g_k^s + (1-\gamma_k^s) d_{k-1}^s, \\
            & (d_k^s)' = (\gamma_k^s)' (g_k^s)' + (1-(\gamma_k^s)') d_{k-1}^s.
        \end{aligned}
    \end{equation}

    \noindent Let $\Delta \gamma_k^s = (\gamma_k^s)' - \gamma_k^s$, then

    \begin{equation}
        \begin{aligned}
            d_k^s - (d_k^s)' & = \gamma_k^s g_k^s + (1-\gamma_k^s) d_{k-1}^s \\ 
            & - (\gamma_k^s + \Delta \gamma_k^s)(g_k^s + \Delta g_k^s) - (1 - \gamma_k^s - \Delta \gamma_k^s) d_{k-1}^s \\
            & = -\gamma_k^s \Delta g_k^s - \Delta \gamma_k^s g_k^s - \Delta \gamma_k^s \Delta g_k^s + \Delta \gamma_k^s d_{k-1}^s. 
        \end{aligned}
    \end{equation}
\noindent With Assumption \ref{Boundedness of subgradient}, Eq. \eqref{boundedness of g} and Eq. \eqref{boundedness of gamma}, we have

\begin{equation}
    ||d_k^s - (d_k^s)'|| = O(\frac{1}{|S_k|}). \Halmos
\end{equation} 
\end{proof}

    For any $\theta^s$, we define

    \begin{equation*}
        \begin{aligned}
            & x_k^s(\theta^s) := \argmin_{x^s \in X} h(x^s, \omega^s) + \langle \lambda_k^s + \theta^s d_k^s, x^s - \bar{x}_k \rangle. \\
            & (x_k^s)'(\theta^s) := \argmin_{x^s \in X} h(x^s, \omega^s) + \langle \lambda_k^s + \theta^s (d_k^s+\Delta d_k^s), x^s - (\bar{x}_k)'\rangle.
        \end{aligned}
    \end{equation*}
    
\begin{lemma} \label{boundedness of L_k^s}

    If a scenario $ \omega^{s_0} \in S_k \backslash S_{k-1}$ added during iteration $k$ is replaced by $(\omega^{s_0})'$, then for any $s \neq s_0$, we have 
    \begin{equation*}
        \begin{aligned}
            & ||x_k^s (\theta^s) - (x_k^s)'(\theta^s)|| = O(\frac{1}{|S_k|}), \ \textit{and} \\
            & |L_k^s (\lambda_k^s + \theta_k^s d_k^s; \omega^s) - L_k^s(\lambda_k^s + (\theta_k^s)' (d_k^s)'; \omega^s)| = O(\frac{1}{|S_k|}).
        \end{aligned}     
    \end{equation*}
\end{lemma}

\begin{proof}{Proof}

    \noindent Note that we have only changed the objective coefficient of $x^s$ and the difference is $ \theta^s \cdot \Delta d_k^s$. According to the sensitivity analysis in linear programming, $||\Delta d_k^s|| = O(\frac{1}{|S_k|})$ and $\theta^s \in [\delta_{min},\delta_{max}]$, we have

    \begin{equation} \label{x_k^s boundedness}
        ||x_k^s(\theta^s) - (x_k^s)'(\theta^s)|| = O(\theta^s \cdot ||\Delta d_k^s||) = O(\frac{1}{|S_k|}). 
    \end{equation}

    \noindent Also, for any $\theta^s$, consider

    \begin{equation*}
        \begin{aligned}
            & g(\theta^s; \omega^s) :=  h(x_k^s(\theta^s), \omega^s) + \langle \lambda_k^s + \theta^s d_k^s, x_k^s(\theta^s) - \bar{x}_k \rangle. \\
            & g'(\theta^s; \omega^s) :=  h((x_k^s)'(\theta^s), \omega^s) + \langle \lambda_k^s + \theta^s (d_k^s)', (x_k^s)'(\theta^s) - \bar{x}_k' \rangle. 
        \end{aligned}
    \end{equation*}

    \noindent Let $\Delta x_k^s (\theta^s) = x_k^s(\theta^s) - (x_k^s)'(\theta^s)$ and $\Delta \bar{x}_k = \bar{x}_k'-\bar{x}_k$, then we have
    \begin{equation*}
        \begin{aligned}
            g(\theta^s; \omega^s) - g'(\theta^s; \omega^s) & = h(x_k^s(\theta^s), \omega^s) + \langle \lambda_k^s + \theta^s d_k^s, x_k^s(\theta^s) - \bar{x}_k \rangle \\
            & -(h((x_k^s)'(\theta^s), \omega^s) + \langle \lambda_k^s + \theta^s (d_k^s)', (x_k^s)'(\theta^s) - \bar{x}_k' \rangle)\\
            & = h(x_k^s(\theta^s), \omega^s) - h(x_k^s(\theta^s) + \Delta x_k^s (\theta^s); \omega^s) \\
            & + \langle \lambda_k^s + \theta^s d_k^s, x_k^s(\theta^s) - \bar{x}_k \rangle \\
            & - \langle \lambda_k^s + \theta^s (d_k^s + \Delta d_k^s), x_k^s(\theta^s) + \Delta x_k^s(\theta^s)  - \bar{x}_k + \Delta \bar{x}_k \rangle) \\
            & = h(x_k^s(\theta^s), \omega^s) - h(x_k^s(\theta^s) + \Delta x_k^s (\theta^s), \omega^s) \\
            & - \langle \lambda_k^s, \Delta x_k^s(\theta^s) \rangle + \langle \lambda_k^s, \Delta \bar{x}_k \rangle - \langle \theta^s d_k^s, \Delta x_k^s(\theta^s) \rangle + \langle \theta^s d_k^s,  \Delta \bar{x}_k \rangle \\
            & - \langle \theta^s \Delta d_k^s, x_k^s(\theta^s) + \Delta x_k^s(\theta^s)  - \bar{x}_k + \Delta \bar{x}_k \rangle  
        \end{aligned}      
    \end{equation*}
    \noindent With Assumption \ref{Lipschitz continuity}, Lemma \ref{boundedness lemma 1} and Eq. \eqref{x_k^s boundedness}, we have
    \begin{equation} \label{boundedness of g function}
        || g(\theta^s; \omega^s) - g'(\theta^s; \omega^s)|| = O(\frac{1}{|S_k|}).
    \end{equation}
    \noindent Since Eq. \eqref{boundedness of g function} is true for any $\theta^s$, we have
    \begin{equation*}
        \begin{aligned}
            & \ | L_k^s (\lambda_k^s + \theta_k^s d_k^s; \omega^s) - L_k^s(\lambda_k^s + (\theta_k^s)' (d_k^s)'; \omega^s)|\\
            = & \ |\max_{\theta^s} g(\theta^s; \omega^s) - \max_{\theta^s} g'(\theta^s; \omega^s)| \\
            \leq & \ \max_{\theta^s} | g(\theta^s; \omega^s) - g'(\theta^s; \omega^s)|\\
            = & \ O(\frac{1}{|S_k|}). \Halmos
        \end{aligned}   
    \end{equation*}
\end{proof}

\begin{theorem} \label{concentration inequality}
    Let $L (\lambda^1,...,\lambda^{|S_k|}) := E_{\{\omega^1,...,\omega^{|S_k|}\}}  [\frac{1}{|S_k|} \sum_{s = 1}^{|S_k|} L_k^s(\lambda^s; \omega^s)]$
    and \\
    $L_k (\lambda^1,...,\lambda^{|S_k|}) := \frac{1}{|S_k|} \sum_{s = 1}^{|S_k|} L_k^s(\lambda^s; \omega^s)$. Let $M_1$ be a large constant, for any $0<\varepsilon<1$, $\kappa>\frac{4L}{\delta_{min}}$, and
        
    \begin{equation} \label{sample complexity}
         |S_k| \geq -8\log(\varepsilon/2)\cdot \frac{M_1^2}{\kappa^2\delta_k^4},
    \end{equation}
    
    \noindent we have
    \begin{equation*}
        \mathbb{P}[|L (\lambda^1,...,\lambda^{|S_k|}) - L_k (\lambda^1,...,\lambda^{|S_k|})| \leq \kappa \delta_k^2, \ \forall \lambda^s \in \mathcal{B}(\lambda_k^s,\delta_k) | \mathcal{L}_{k-1}] \geq 1-\varepsilon. 
    \end{equation*}
\end{theorem}

\begin{proof}{Proof}
    For the points $\lambda_k^s$, we will first show that given $\mathcal{L}_{k-1}$, $L_k$ has the bounded difference property and then use McDiarmid's inequality to prove the theorem. Since $\lambda_k^s$ depends on $\omega^1,\omega^2,...,\omega^{|S_k|}$, the McDiarmid's inequality cannot be applied to $L_k(\lambda_k^1,...,\lambda_k^{|S_k|})$ directly. However, we can express $L_k$ as a function of $\omega^1,\omega^2,...,\omega^{|S_k|}$ and show that $L_k(\omega^1,\omega^2,...,\omega^{|S_k|})$ has the bounded difference property. Specifically, the bounded difference property with respect to $L_k(\omega^1,\omega^2,...,\omega^{|S_k|})$ is written as

\begin{equation*}
    |L_k(\omega^1,\omega^2,...,\omega^{s_0},...,\omega^{|S_k|}) - L_k(\omega^1,\omega^2,...,(\omega^{s_0})',...,\omega^{|S_k|})| = O(\frac{1}{|S_k|}).
\end{equation*}

\noindent We will divide the analysis into 2 cases to bound the changes of $L_k^s(\lambda_k^s;\omega^s)$.

\begin{itemize}
    \item Case 1: For $s \neq s_0$. Given $\mathcal{L}_{k-1}$, based on Lemma \ref{boundedness of L_k^s}, we have
    \begin{equation*}
    |L_k^s(\lambda_k^s;\omega^s) - L_k^s((\lambda_k^s)';\omega^s)| = O(\frac{1}{|S_k|}).
    \end{equation*}
    \item Case 2: For $s = s_0$. Given $\mathcal{L}_{k-1}$, based on Assumption \ref{boundedness of dual objective}, we have
    \begin{equation*}
    |L_k^s(\lambda_k^s;\omega^s) - L_k^s((\lambda_k^s)';(\omega^{s})')|\leq 2M.
    \end{equation*}
\end{itemize}

\noindent Thus,
    \begin{equation*}
        \begin{aligned}
            & |L_k(\omega^1,\omega^2,...,\omega^{s_0},...,\omega^{|S_k|}) - L_k(\omega^1,\omega^2,...,(\omega^{s_0})',...,\omega^{|S_k|})| \\
            \leq & \ \frac{\sum_{s \neq s_0}|L_k^s(\lambda_k^s;\omega^s) - L_k^s((\lambda_k^s)';\omega^s)|}{|S_k|} + \frac{|L_k^{s_0}(\lambda_k^{s_0};\omega^{s_0}) - L_k^{s_0}((\lambda_k^{s_0})';(\omega^{s_0})')|}{|S_k|} \\
            =  & \ O(\frac{1}{|S_k|}).
        \end{aligned}
    \end{equation*}
\noindent Let the bounded difference be specificed by $\frac{M_1}{|S_k|}$. Applying McDiarmid's inequality,
\begin{equation*}
        \quad \ \mathbb{P}(|L (\lambda_k^1,...,\lambda_k^{|S_k|}) - L_k (\lambda_k^1,...,\lambda_k^{|S_k|})| \leq \frac{1}{2}\kappa \delta_k^2)
        \geq 1-2\exp(-\frac{\kappa^2\delta_k^4 |S_k|^2}{8 \cdot |S_k| \cdot M_1^2}) \geq 1-\varepsilon,
\end{equation*}

\noindent which indicates that when 
\begin{displaymath}
   |S_k| \geq -8\log(\varepsilon/2)\cdot \frac{M_1^2}{\kappa^2\delta_k^4}, 
\end{displaymath} 
we have 
\begin{equation*}
    \mathbb{P}(|L (\lambda_k^1,...,\lambda_k^{|S_k|}) - L_k (\lambda_k^1,...,\lambda_k^{|S_k|})| \leq \frac{1}{2}\kappa \delta_k^2) \geq 1-\varepsilon.
\end{equation*}

\noindent For any other $\lambda^s \in \mathcal{B}(\lambda_k^s,\delta_k)$, if $|L (\lambda_k^1,...,\lambda_k^{|S_k|}) - L_k (\lambda_k^1,...,\lambda_k^{|S_k|})| \leq \frac{1}{2}\kappa \delta_k^2$, then

\begin{equation*}
    \begin{aligned}
        & \quad \ |L (\lambda^1,...,\lambda^{|S_k|}) - L_k (\lambda^1,...,\lambda^{|S_k|})| \\
        & \leq |L (\lambda^1,...,\lambda^{|S_k|}) - L(\lambda_k^1,...,\lambda_k^{|S_k|})| \\
        & + |L (\lambda_k^1,...,\lambda_k^{|S_k|}) - L_k (\lambda_k^1,...,\lambda_k^{|S_k|})  | + |L_k (\lambda_k^1,...,\lambda_k^{|S_k|}) - L_k (\lambda^1,...,\lambda^{|S_k|})| \\
        & \leq 2L \cdot \delta_k + \frac{1}{2} \kappa \delta_k^2\\
        & \leq \kappa \delta_k^2,
    \end{aligned}
\end{equation*}

\noindent where the second last inequality is due to the Assumption \ref{Lipschitz continuity} that $L_k$ and $L_k^s$ are Lipschitz continuous and the last inequality is because $\kappa>\frac{4L}{\delta_{min}}>\frac{4L}{\delta_k}$. Thus, we conclude that if $|S_k|$ satisfies Equation \eqref{sample complexity}, then 

\begin{equation*} \label{fl1}
     \mathbb{P}[|L (\lambda^1,...,\lambda^{|S_k|}) - L_k (\lambda^1,...,\lambda^{|S_k|})| \leq \kappa \delta_k^2, \ \forall \lambda^s \in \mathcal{B}(\lambda_k^s,\delta_k)] \geq 1-\varepsilon. \Halmos 
\end{equation*}

\end{proof}

\subsection{Sufficient Increase and Submartingale Property}

\begin{definition} Let $\theta_k \odot d_k := (\theta_k^1d_k^1,...,\theta_k^{|S_k|}d_k^{|S_k|})$. The value estimates $ L_k := L_k(\lambda_k)$ and $ L_{k+1} = L_k(\lambda_k + \theta_k \odot d_k)$ are $\varepsilon_F$-accurate estimates of $ L(\lambda_k)$ and $ L(\lambda_k + \theta_k \odot d_k)$, respectively, for a given $\delta_k$ if and only if 
\begin{equation*}
    |L_k - L(\lambda_k)| \leq \varepsilon_F \delta_k^2, \quad |L_{k+1}-L(\lambda_k + \theta_k \odot d_k)| \leq \varepsilon_F \delta_k^2.
\end{equation*}
\end{definition}

\begin{definition} \label{stocahstic accuracy} A sequence of model estimates $\{L_k, L_{k+1}\}$ is said to be $\mu_2$-probabilistically $\varepsilon_F$-accurate with respect to $\{\Lambda_k,\Delta_k,S_k\}$ if the event 
\begin{equation}
    J_k \overset{\Delta}{=}  \{  L_k, L_{k+1} \ \textit{are} \ \varepsilon_F\textit{-accurate estimates }  \ \textit{for} \ \Delta_k \}
\end{equation}
\noindent satisfies the condition 
\begin{equation}
    \mathbb{P}(J_k|\mathcal{L}_{k-1}) \geq \mu_2 >0.
\end{equation}
\end{definition}

\begin{lemma} \label{d1}
Suppose that ${L}_k$ is a $\kappa$-approximation of $L$ on $\mathcal{B}(\lambda_k^s ,\delta_k)$. If 
\begin{equation} \label{d1c}
    \delta_k \leq \frac{m_1}{4n\kappa} ||d_k|| \quad \textit{and} \quad C_1=\frac{m_1}{2n},
\end{equation}

\noindent then there exist suitable step sizes $\theta_k^s \in L$ and $\theta_k^s ||d_k^s|| > \frac{\delta_k}{n}$ such that 
\begin{equation} \label{d1con}
    L(\lambda_k + \theta_k \odot d_k)-L(\lambda_k) \geq C_1 ||d_k|| \delta_k.
\end{equation}

\end{lemma}
\begin{proof}{Proof}
Since $\theta_k^s \in L$ and $\theta_k^s ||d_k^s|| > \frac{\delta_k}{n}$, for each $L_k^s$, 
\begin{equation*}
    L_k^s(\lambda_k^s + \theta_k^s d_k^s)-L_k^s(\lambda_k^s) \geq m_1||d_k^s||^2 \theta_k^s \geq \frac{m_1}{n}||d_k^s|| \delta_k.
\end{equation*}
\noindent Thus,
\begin{equation*}
    \begin{aligned}
        L_k(\lambda_k + \theta_k \odot d_k) -  L_k(\lambda_k) & =  \frac{1}{|S_k|}\sum_{s=1}^{|S_k|}[L_k^s(\lambda_k^s + \theta_k^s d_k^s)-L_k^s(\lambda_k^s)] \\
        & \geq \frac{m_1}{n}  ||d_k|| \delta_k
    \end{aligned}
\end{equation*}
\noindent Also, since $L_k$ is $\kappa$-approximation of $L$ on $\mathcal{B}(\lambda_k^s,\delta_k)$, we have 
\begin{equation}
    \begin{aligned}
          \quad \ \ L(\lambda_k + \theta_k \odot d_k)-L(\lambda_k) 
         & = L(\lambda_k + \theta_k \odot d_k)-L_k(\lambda_k + \theta_k \odot d_k) \\
         & + L_k(\lambda_k + \theta_k \odot d_k) -L_k(\lambda_k)+L_k(\lambda_k)-L(\lambda_k)\\
         & \geq -2\kappa \delta_k^2+\frac{m_1}{n}  ||d_k||\delta_k \\
         & \geq C_1||d_k||\delta_k,
    \end{aligned} 
\end{equation}
which concludes the proof. \Halmos
\end{proof}

\begin{lemma} \label{d2} Suppose that $L_k$ is $\kappa$-approximation of $L$ on the ball $\mathcal{B}(\lambda_k^s,\delta_k)$ and the estimates $(L_k,L_{k+1})$ are $\varepsilon_F$-accurate with $\varepsilon_F \leq \kappa$. If $\delta_k \leq 1$ and 
\begin{equation} \label{d2c}
    \delta_k \leq \min \{\frac{(1-\eta_1)C_1 ||d_k||}{2\kappa}, \frac{||d_k||}{\eta_2}\},
\end{equation}

\noindent then a suitable step size $\theta_k^s \in L$ and $\theta_k^s ||d_k^s|| > \frac{\delta_k}{n}$ makes the $k$-th iteration successful.
\end{lemma}

\begin{proof}{Proof}
    Since $\theta_k^s \in L$ and $\theta_k^s ||d_k^s|| > \frac{\delta_k}{n}$, from Lemma \ref{d1},
    \begin{equation*}
    L_k(\lambda_k + \theta_k \odot d_k)-L_k(\lambda_k) \geq 2 C_1 ||d_k|| \delta_k
    \end{equation*}
    
    \noindent Also, $L_k$ is $\kappa$-approximation of $L$ on $\mathcal{B}(\lambda_k^s,\delta_k)$ and the estimates $(L_{k},L_{k+1})$ are $\varepsilon_F$-accurate with $\varepsilon_F \leq \kappa$ implies that 
    
    \begin{equation*}
        \begin{aligned}
            \rho_k & = \frac{L_{k}-L_{k-1}}{L_{k+1}-L_k}\\
            & = \frac{L_{k}-L(\lambda_k + \theta_k \odot d_k)}{L_{k+1}-L_k}+\frac{L(\lambda_k + \theta_k \odot d_k)-L_{k+1}}{L_{k+1}-L_k}+\frac{L_{k+1}-L_k}{L_{k+1}-L_k}\\
            & + \frac{L_k-L(\lambda_k)}{L_{k+1}-L_k}+\frac{L(\lambda_k)-L_{k-1}}{L_{k+1}-L_k},\\
        \end{aligned}
    \end{equation*}
    
\noindent which indicates that
\begin{equation*}
    1-\rho_k \leq \frac{ 2\kappa\delta_k}{C_1||d_k||}\leq 1-\eta_1.
\end{equation*}

\noindent Thus, we have $\rho_k \geq \eta_1$, $||d_k||>\eta_2\delta_k$ and the iteration is successful. \Halmos
\end{proof}

\begin{lemma} \label{d3} Suppose the function value estimates $\{(L_{k},L_{k+1})\}$ are $\varepsilon_F$-accurate and
\begin{equation*} \label{d3c}
    \varepsilon_F < \eta_1 \eta_2 C_1.
\end{equation*}
\noindent If the $k$th iteration is successful, then the improvement in $L$ is bounded such that 

\begin{equation*}
    L(\lambda_k + \theta_k \odot d_k)-L(\lambda_k) \geq C_2\delta_k^2,
\end{equation*}

\noindent where $C_2 \overset{\Delta}{=} 2\eta_1 \eta_2 C_1-2\varepsilon_F$.
\end{lemma}

\begin{proof}{Proof}
    Since $\theta_k^s$ satisfies the L condition in Equation \eqref{strong-wolfe condition} and $\theta_k^s ||d_k^s|| > \frac{\delta_k}{n}$, from Lemma \ref{d1},
    \begin{equation*}
    L_k(\lambda_k + \theta_k \odot d_k)-L_k(\lambda_k) \geq 2 C_1 ||d_k|| \delta_k.
    \end{equation*}
    Also, since the iteration is successful, we have $||d_k||>\eta_2 \delta_k$. Thus, we have
    
    \begin{equation*}
        L_{k}-L_{k+1} \geq \eta_1(L_k - L_{k-1}) \geq 2\eta_1 C_1 ||d_k||\delta_k \geq 2\eta_1 \eta_2 C_1 \delta_k^2.
     \end{equation*}
    Then, since the estimates are $\varepsilon_F$ -accurate, we have that the improvement in $f$ can be bounded as  
    \begin{displaymath}
        \begin{aligned}
            & L(\lambda_k + \theta_k \odot d_k)-L(\lambda_k) \\
            = & L(\lambda_k + \theta_k \odot d_k)-L_{k+1}+L_{k+1}-L_{k}+L_{k}-L(\lambda_k ) \\
            \geq & C_2 \delta_k^2. \Halmos 
        \end{aligned}
    \end{displaymath}
\end{proof}

\begin{theorem} \label{delta 1} Let the random function $V_k \overset{\Delta}{=} \mathcal{L}_{k+1}-\mathcal{L}_{k}$, the corresponding realization be $v_k$, 

\begin{equation} \label{xi}
    \zeta = \max\{4n\kappa,\eta_2,\frac{4\kappa}{C_1(1-\eta_1)}\}  \quad and \quad  \mu_1\mu_2 > \frac{1}{2}.
\end{equation}

\noindent  Then for any $\varepsilon>0$ such that $||d_k|| \geq \varepsilon$ and $\zeta \Delta_k \leq \varepsilon$, we have 
\begin{equation*}
    \mathbb{E}[V_k \big| \ ||d_k|| \geq \varepsilon, \zeta \Delta_k \leq \varepsilon] \geq \frac{1}{2}C_1 ||d_k|| \Delta_k \geq \theta \varepsilon \Delta_k,
\end{equation*}
\noindent where $\theta = \frac{1}{2}  C_1$.
\end{theorem} 

\begin{proof}{Proof}
First of all, if $k$th iteration is successful, i.e. $\lambda_{k+1}=\lambda_k + \theta_k \odot d_k$, we have 
\begin{equation} \label{s}
    v_k = L(\lambda_k + \theta_k \odot d_k)-L(\lambda_k).
\end{equation}

\noindent If $k$th iteration is unsuccessful, i.e. $\lambda_{k+1}=\lambda_k$ we have

\begin{equation} \label{b1}
    v_k = L(\lambda_k)-L(\lambda_k)=0.
\end{equation}

\noindent Then we will divide the analysis into 4 cases according to the states (true/false) observed for the pair $(I_k, J_k)$.

\noindent (a) $I_k$ and $J_k$ are both true. Since the $L_k$ is a $\kappa$-approximation of $L$ on $\mathcal{B}(\lambda_k^s ,\delta_k)$ and condition (\ref{d1c}) is satisfied, Lemma \ref{d1} holds. Also, since the estimates $(L_k,L_{k+1})$ are $\varepsilon_F$-accurate and condition (\ref{d2c}) is satisfied, we have Lemma \ref{d2} holds.  Combining (\ref{d1con}) with (\ref{s}), we have 

\begin{equation} \label{b2}
    v_k \geq  C_1 ||d_k||\delta_k \overset{\Delta}{=} b_k^1.
\end{equation}

\noindent (b) $I_k$ is true but $J_k$ is false. Since $L_k$ is a $\kappa$-approximation of $L$ on $\mathcal{B}(\lambda_k^s ,\delta_k)$ and condition (\ref{d1c}) is satisfied, it follows that Lemma \ref{d1} still holds. If the iteration is successful, we have (\ref{b2}), otherwise we have (\ref{b1}). Thus, we have
$v_k \geq 0$.

\noindent (c) $I_k$ is false but $J_k$ is true. If the iteration is successful, since the estimates $(L_k,L_{k+1})$ are $\varepsilon_F$-accurate and condition (\ref{d3c}) is satisfied, Lemma \ref{d3} holds. Hence, 
\begin{equation*}
    v_k \geq C_2 \delta_k^2.
\end{equation*}

\noindent If the iteration is unsuccessful, we have (\ref{b1}). Thus, we have $v_k \geq 0$ whether the iteration is successful or not.

\noindent (d) $I_k$ and $J_k$ are both false. Since $L$ is convex and $\theta_k^s||d_k^s||<\delta_k$, for any $g(\theta_k) \in \partial L(\lambda_k + \theta_k \odot d_k)$, with Assumption  \ref{Boundedness of subgradient}, $||g(\theta_k)|| \leq G $ , we have 

\begin{equation*}
    v_k=L(\lambda_k+ \theta_k \odot d_k)-L(\lambda_k) \geq -\langle g(\theta_k), \theta_k \odot d_k \rangle \geq - G \delta_k.
\end{equation*}

\noindent If the iteration is successful, then
\begin{equation*}
    v_k \geq -G \delta_k \overset{\Delta}{=} b_2 .
\end{equation*}

\noindent If the iteration is unsuccessful, we have (\ref{b1}). Thus, we have $v_k \geq  b_2$ whether the iteration is successful or not.

With the above four cases, we can bound $\mathbb{E}[V_k \big| \ ||d_k|| \geq \varepsilon, \zeta \Delta_k \leq \varepsilon]$ based on different outcomes of $I_k$ and $J_k$. Let $B_1$ and $B_2$ be the random counterparts of $b_1$ and $b_2$.  Then we have 
\begin{equation*}
    \begin{aligned}
         & \quad \  \mathbb{E}[V_k \big| \ ||d_k|| \geq \varepsilon, \zeta \Delta_k \leq \varepsilon] \\
         & \geq \mu_1 \mu_2 B_1 + (\mu_1(1-\mu_2)+\mu_2(1-\mu_1))\cdot 0 + (1-\mu_1)(1-\mu_2) B_2 \\
         & = \mu_1 \mu_2  (C_1 ||d_k||\Delta_k) - (1-\mu_1)(1-\mu_2)G \Delta_k. \\
    \end{aligned}
\end{equation*}

\noindent Choose $\mu_1 \in (1/2,1)$ and $\mu_2 \in (1/2,1)$ large enough such that 

\begin{equation*}
 \mu_1 \mu_2 \geq 1/2 \quad \text{and \quad }\frac{\mu_1 \mu_2}{(1-\mu_1)(1-\mu_2)} \geq \frac{2 G}{C_1||d_k||},
\end{equation*}

\noindent we have
\begin{displaymath}
    \begin{aligned}
    \mathbb{E}[V_k \big| \ ||d_k|| \geq \varepsilon, \zeta \Delta_k \leq \varepsilon] & \geq  \frac{1}{2} C_1 ||d_k||\Delta_k. \Halmos 
    \end{aligned}
\end{displaymath}
\end{proof}

\noindent The original deterministic version of PHA generates a sequence of decisions whose convergence is characterized by a contraction to the set of optimal solutions of the instance under consideration. This property leads to the convergence of the sequence of points generated by the PH algorithm. In order to demonstrate that the stochastic PH algorithm also possesses an equivalent type of contraction, although the sampling-based nature of the new algorithm leads to the contraction property in expectation.  This is the stochastic analog of the contraction property exhibited by the original deterministic PHA.

\begin{theorem} \label{contraction obj} Given $||d_k|| \geq \varepsilon$ and $\delta_{max} = \frac{\varepsilon}{\zeta}$, we have the contraction for the objective function values in expectation.

\begin{equation} \label{contraction of objective value}
    E[L(\lambda^*) - L(\lambda_{k+1})] \leq (1-\eta) \cdot E[L(\lambda^*) - L(\lambda_k)],
\end{equation}

\noindent where $\eta = \frac{C_1 \delta_{min} \varepsilon}{M}$.

\end{theorem}

\begin{proof}{Proof}
    Based on Assumption \ref{boundedness of dual objective}, 

    \begin{equation*}
        ||d_k|| \geq \varepsilon \geq\varepsilon\frac{E[L(\lambda^*) - L(\lambda_k)]}{2M}.
    \end{equation*}
    
    \noindent By combining Theorem \ref{delta 1}, if $||d_k|| \geq \varepsilon$ and $\delta_{max} = \frac{\varepsilon}{\zeta}$, we have

    \begin{equation*}
        E[L(\lambda_{k+1}) - L(\lambda_k)] \geq \frac{1}{2} C_1 \Delta_k ||d_k|| \geq - \frac{C_1 \delta_{min} \varepsilon}{M} E[L(\lambda_k) - L(\lambda^*)].
    \end{equation*}

    \noindent Thus, by adding $-L(\lambda^*)$ on both side of the inequality,

    \begin{equation*}
        E[L(\lambda_{k+1}) - L(\lambda^*)] \geq (1-\eta) \cdot  E[L(\lambda_k) - L(\lambda^*)].
    \end{equation*}

    \noindent Or equivalently,

    \begin{equation*}
        E[L(\lambda^*) - L(\lambda_{k+1})] \leq (1-\eta) \cdot E[L(\lambda^*) - L(\lambda_k)]. \Halmos
    \end{equation*}
 
\end{proof}

\begin{theorem} \label{contraction dec}
 If the dual objective function value sequence is a contraction in expectation, then it implies that the dual decision variable sequence is also a contraction in expectation.

 \begin{equation}
     E[||\lambda_{k+1} - \lambda^*||] \leq \sqrt{(1-\eta)} \cdot E[||\lambda_k - \lambda^*||].
 \end{equation}
\end{theorem}

\begin{proof}{Proof}
    First, note that the dual objective function (with regularization term) is a strongly-concave function with the parameter $m = 2$. Thus, we have

    \begin{equation} \label{strong concavity}
        \begin{aligned}
        E [L(\lambda^*) - L(\lambda_{k+1})] & \geq E[\langle \nabla L(\lambda^*), \lambda^* - \lambda_{k+1} \rangle + \frac{m}{2} ||\lambda^* - \lambda_{k+1}||^2] \\
        & = E [||\lambda^* - \lambda_{k+1}||^2],
        \end{aligned}      
    \end{equation}

    \noindent where the equality is because $\nabla L(\lambda^*) = 0$ and $m = 2$.

    \noindent On the other hand, $L$ is also a Lipschitz continuous function with Lipschitz constant equals 1. Thus, combining with the concave property of $L$ and using the triangle inequality, we obtain

    \begin{equation} \label{concave and Lipschitz}
        \begin{aligned}
            E[L(\lambda^*) - L(\lambda_k)] & \leq E[\langle \nabla L(\lambda_k), \lambda^* - \lambda_k \rangle] \\
            & \leq E[\langle \nabla L(\lambda^*) - \nabla L(\lambda_k), \lambda^* - \lambda_k \rangle] \\
             & \leq E[||\nabla L(\lambda^*) - \nabla L(\lambda_k)|| \cdot ||\lambda^* - \lambda_k||] \\
            & \leq E[||\lambda^* - \lambda_k|| \cdot ||\lambda^* - \lambda_k||] \\
            & = E[||\lambda^* - \lambda_k||^2].
        \end{aligned}  
    \end{equation}

    \noindent Finally, \eqref{contraction of objective value}, \eqref{strong concavity} and \eqref{concave and Lipschitz} together imply that

    \begin{equation*}
        \begin{aligned}
            E [||\lambda^* - \lambda_{k+1}||^2] & \leq E[L(\lambda^*) - L(\lambda_{k+1})] \\
            & \leq (1-\eta) \cdot E[L(\lambda^*) - L(\lambda_k)] \\
            & \leq (1-\eta) \cdot  E[||\lambda^* - \lambda_k||^2]
        \end{aligned}
    \end{equation*}

    \noindent Thus, by taking the square root on both sides of the inequality, we get 

    \begin{equation*}
        E[||\lambda_{k+1} - \lambda^*||] \leq \sqrt{(1-\eta)} \cdot E[||\lambda_k - \lambda^*||]. \Halmos
    \end{equation*}   
       
\end{proof}

\noindent \textbf{Remark}: It is important to note that the contraction parameter $\sqrt{1-\eta}$ is a constant that does not depend on the choice of the penalty parameter. This result aligns with the classic PHA, where any penalty parameter can yield a contraction. However, the key difference is that the sampling-based PHA employs a line-search algorithm to first ensure a sufficient increase in the dual objective function values, and then leveraging the strong concavity to establish the contraction within the dual space. In contrast, the classic PHA operates directly within the primal-dual space.

\noindent A quick observation of the concluding condition reveals that the requirement $\Delta_k \leq \frac{\varepsilon}{\zeta}$ implies that the supermartingale property holds. This prompts us to impose the condition that restricts $\delta_{max} = \frac{\varepsilon}{\zeta}$. This property is formalized in the following corollary. 

\begin{corollary} \label{submartingale} Let 
\begin{equation*}
    T_{\varepsilon} = \inf \{ k \geq 0: ||d_k|| < \varepsilon\}.
\end{equation*}

\noindent Then $T_{\varepsilon}$ is a stopping time for the stochastic process $\Lambda^k$. Moreover, conditioned on $T_{\varepsilon} \geq k$, $\{\mathcal{L}_k\}$ is a submartingale. 

\end{corollary}

\begin{proof}{Proof}

{    
    From Theorem \ref{delta 1}, we have
    \begin{equation} \label{super 1}
        \mathbb{E}[\mathcal{L}_k | \mathbb{L}_{k-1}, T_{\varepsilon}  > k] \geq L_{k-1} + \Theta \varepsilon \Delta_k.
    \end{equation}
}

\noindent Hence, $\mathcal{L}_k$ is a submartingale. \Halmos
\end{proof}
\subsection{Convergence Rate} \label{4.3}

Building upon the results established in Theorem \ref{delta 1}, where $\mathcal{L}_k$ is demonstrated to be a submartingale and $T_{\varepsilon}$ is identified as a stopping time, we proceed to construct a renewal reward process for analyzing the bound on the expected value of $T_{\varepsilon}$.
As highlighted in the abstract, Theorem \ref{Convergence rate} confirms the rate of convergence to be $O(1/\varepsilon^2)$.
To begin with, let us define the renewal process $\{ A_l \}$ as follows: set $A_0 = 0$, and for each $l > 0$, define $A_l = \inf \{ m > A_{l-1} : \zeta \Delta_m \geq \varepsilon \}$, with $\zeta$ being specified in \eqref{xi}. Additionally, we define the inter-arrival times $\tau_l = A_l - A_{l-1}$. Lastly, we introduce the counting process $N(k) = \max \{n: A_l \leq k\}$, representing the number of renewals occurring up to the $k^{th}$ iteration.
%Readers seeking the in-depth proofs of Theorem \ref{Convergence rate} and Theorem \ref{optimality} can refer to Appendix \ref{convergence rate proof} and \ref{optimality condition proof}, respectively.

\begin{lemma} \label{tau} Let $ \frac{1}{2} < p=\mu_1 \mu_2 \leq \mathbb{P}(I_k \cap J_k)$. Then for all $l \geq 1$,

\begin{equation*}
    \mathbb{E}[\tau_l] \leq \frac{p}{2p-1}.
\end{equation*}

\end{lemma}

\begin{proof}{Proof}
First,
    \begin{equation*}
        \begin{aligned}
            \mathbb{E}[\tau_l] & = \mathbb{E}[\tau_l | \zeta \Delta_{A_{l-1}} > \varepsilon] \cdot \mathbb{P}(\zeta \Delta_{A_{l-1}} > \varepsilon) +  \mathbb{E}[\tau_l | \zeta \Delta_{A_{l-1}} = \varepsilon] \cdot \mathbb{P}(\zeta \Delta_{A_{l-1}} = \varepsilon)\\
            & \leq \max \{ \mathbb{E}[\tau_l | \zeta \Delta_{A_{l-1}} > \varepsilon], \mathbb{E}[\tau_n | \zeta \Delta_{A_{l-1}} = \varepsilon] \}.
        \end{aligned}
    \end{equation*}
    
\noindent If $\zeta \Delta_{A_{l-1}} > \varepsilon$, according to Algorithm \ref{SPHA}, 

\begin{equation} \label{tau 1}
    \mathbb{E}[\tau_l | \zeta \Delta_{A_{l-1}} > \varepsilon] = 1.
\end{equation}

\noindent If $\zeta \Delta_{A_{l-1}} = \varepsilon$, by Theorem \ref{delta 1}, we can treat $\{ \Delta_{A_{l-1}},...,\Delta_{A_{l}}\}$ as a random walk, and we have 

\begin{equation} \label{tau 2}
    \mathbb{E}[\tau_l | \zeta \Delta_{A_{l-1}} = \varepsilon] \leq \frac{p}{2p-1}.
\end{equation}

\noindent Combining (\ref{tau 1}) and (\ref{tau 2}) completes the proof. \Halmos

\end{proof}

\begin{lemma} \label{N(T)} Let $\zeta$ and $\theta$ be the same as in Theorem \ref{delta 1} and $\Delta_{max} = \frac{\varepsilon}{\zeta}$, then
\begin{equation*}
    \mathbb{E}[N(T_{\varepsilon})] \leq \frac{2 \zeta M }{\theta \varepsilon^2}+\frac{\zeta \Delta_{max}}{\varepsilon} = \frac{2 \zeta M }{\theta \varepsilon^2}+1,
\end{equation*}

\noindent where $M$ is defined in Assumption \ref{boundedness of dual objective}.
\end{lemma}

\begin{proof}{Proof}
    We will first show that 
    \begin{equation} \label{R_k}
        R_{k \wedge T_{\varepsilon}} =\mathcal{L}_{k \wedge T_{\varepsilon}} + \Theta \varepsilon \sum_{j=0}^{k \wedge T_{\varepsilon}} \Delta_j
    \end{equation}
\noindent is a submartingale. Using (\ref{super 1}),

\begin{equation} \label{Supermar. proof}
    \begin{aligned}
        \mathbb{E}[R_{k \wedge T_{\varepsilon}}|\mathcal{L}_{k-1}] & = \mathbb{E}[\mathcal{L}_{k \wedge T_{\varepsilon}}|\mathbb{L}_{k-1}] + \mathbb{E}[\Theta \varepsilon \sum_{j=0}^{k \wedge T_{\varepsilon}} \Delta_j | \mathbb{L}_{k-1} ]\\
        & \leq \mathcal{L}_{k-1} - \Theta \varepsilon \Delta_k + \mathbb{E}[\Theta \varepsilon \sum_{j=0}^{k \wedge T_{\varepsilon}} \Delta_j | \mathbb{L}_{k-1} ] \\
        & = \mathcal{L}_{k-1} + \mathbb{E}[\Theta \varepsilon \sum_{j=0}^{(k-1) \wedge T_{\varepsilon}} \Delta_j | \mathbb{L}_{k-1} ] \\
        & = \mathcal{L}_{k-1} + \Theta \varepsilon \sum_{j=0}^{(k-1) \wedge T_{\varepsilon}} \Delta_j \\
        & = R_{(k-1) \wedge T_{\varepsilon}},
    \end{aligned}
\end{equation}

\noindent where the summation in the last expectation in \eqref{Supermar. proof} is true by moving $\Theta \varepsilon \Delta_k$ inside the summation so that it has one less term if $k < T_{\varepsilon}$.

\noindent If $k < T_{\varepsilon}$, then 

\begin{equation*}
    |R_{k \wedge T_{\varepsilon}}| = |R_k| \leq M + \Theta \varepsilon k \Delta_{max}.
\end{equation*}

\noindent If $k \geq T_{\varepsilon}$, then 

\begin{equation*}
    |R_{k \wedge T_{\varepsilon}}| = |R_{\varepsilon}| \leq  M + \Theta \varepsilon T_{\varepsilon} \Delta_{max}.
\end{equation*}

\noindent This is also bounded almost surely since $T_{\varepsilon}$ is bounded almost surely. Hence, according to \eqref{R_k} and the optional stopping theorem~\cite{GGD2020}, we have

\begin{equation} \label{optinal 1}
    \mathbb{E}[\Theta \varepsilon \sum_{j=0}^{ T_{\varepsilon}} \Delta_j] \leq \mathbb{E}[R_{T_{\varepsilon}}] + M \leq \mathbb{E}[R_0] + M \leq 2 M + \Theta \varepsilon \Delta_{max}.
\end{equation}

\noindent Furthermore, since the renewal $A_n$ happens when $\zeta \Delta_j \geq \varepsilon$ and $N(T_{\varepsilon})$ is a subset of $\{0,1,2,...,T_{\varepsilon}\}$, we have 

\begin{equation} \label{optinal 2}
    \Theta \varepsilon \Big(\sum_{j=0}^{ T_{\varepsilon}} \zeta \Delta_j\Big) \geq \Theta \varepsilon \Big( N(T_{\varepsilon}) \varepsilon \Big).
\end{equation}

\noindent Combining (\ref{optinal 1}) and (\ref{optinal 2}), we have \begin{displaymath}
    \mathbb{E}[N(T_{\varepsilon})] \leq \frac{2 \zeta M + \zeta \Theta \varepsilon \Delta_{max}}{\Theta\varepsilon^2} \leq \frac{2 \zeta M }{\Theta \varepsilon^2}+\frac{\zeta \Delta_{max}}{\varepsilon}. \Halmos
    \end{displaymath}  
\end{proof}

\begin{theorem}\label{Convergence rate}
Under conditions enunciated in Assumptions \ref{boundedness of primal variable}-\ref{Boundedness of subgradient}, we have
    \begin{equation} \label{convergence equation}
    \mathbb{E}[T_{\varepsilon}] \leq \frac{p}{2p-1}\Big(\frac{2 \zeta M }{\Theta \varepsilon^2}+2\Big).
\end{equation}
\end{theorem}

\begin{proof}{Proof}
First, note that $N(T_{\varepsilon})+1$ is a stopping time for the renewal process $\{ A_n: n \geq 0\}$. Thus, using Wald's equation (inequality form) \cite{GGD2020}, we have 
    
\begin{equation*}
        \mathbb{E}[A_{N(T_{\varepsilon})+1}] \leq \mathbb{E}[\tau_1] \mathbb{E}[N(T_{\varepsilon})+1].
\end{equation*}
    
\noindent Moreover, since $A_{N(T_{\varepsilon})+1} \geq T_{\varepsilon}$, we have

\begin{equation*}
    \mathbb{E}[T_{\varepsilon}] \leq \mathbb{E}[\tau_1] \mathbb{E}[N(T_{\varepsilon})+1].
\end{equation*}

\noindent Hence, by Lemma \ref{tau} and Lemma \ref{N(T)}
\begin{displaymath}
    \mathbb{E}[T_{\varepsilon}] \leq \frac{p}{2p-1}\Big(\frac{2 \zeta M }{\Theta \varepsilon^2}+2\Big). \Halmos 
\end{displaymath}

\end{proof}

\subsection{Optimality Condition}

\begin{lemma} \label{d_k^*}
    If $k$ is the smallest index for which $||d_k|| \leq \varepsilon$ and $||d_{k-1}|| > \sqrt{1 + \eta} \cdot \varepsilon $, then we have $||g_k|| \leq \sqrt{1+\frac{1}{\eta}} \cdot \varepsilon$.
\end{lemma}

\begin{proof}{Proof}
    Suppose the claim is false, then $||g_k||^2 > (1+\frac{1}{\eta}) \cdot \varepsilon^2$. Thus,
    \begin{equation} \label{norm of d_k}
        \begin{aligned}
            ||d_k||^2 & = ||\lambda_k^* g_k-(1-\lambda_k^*)d_{k-1}||^2\\
            & = (\lambda_k^*)^2||g_k||^2+(1-\lambda_k^*)^2||d_{k-1}||^2-2\lambda_k^*(1-\lambda_k^*)\langle g_k,d_{k-1} \rangle \\
            & \geq (\lambda_k^*)^2||g_k||^2+(1-\lambda_k^*)^2||d_{k-1}||^2 \\
            & > [(1+\frac{1}{\eta}) \cdot (\lambda_k^*)^2 + (1+\eta) \cdot (1-\lambda_k^*)^2 ] \cdot \varepsilon^2,
        \end{aligned}
    \end{equation}

\noindent where the first inequality holds because of the condition R in \eqref{strong-wolfe condition}. To prove that $||d_k|| > \varepsilon$, it suffices to show that 

\begin{equation*}
    (1+\frac{1}{\eta}) \cdot (\lambda_k^*)^2 + (1+\eta) \cdot (1-\lambda_k^*)^2 \geq 1.
\end{equation*}

\noindent Note that 

\begin{equation*}
    \begin{aligned}
        p(\lambda_k^*) \overset{\Delta}{=} & (1+\frac{1}{\eta}) \cdot (\lambda_k^*)^2 + (1+\eta) \cdot (1-\lambda_k^*)^2 - 1 \\
        = &(2+\frac{1}{\eta} + \eta) \cdot (\lambda_k^*)^2 -  2(1+\eta) \lambda_k^* + \eta, 
    \end{aligned}   
\end{equation*}

\noindent is a quadratic function with respect to $\lambda_k^*$ and its discriminant is $0$. Therefore, for any $\lambda_k^* \in [0,1]$, we have $p(\lambda_k^*) \geq 0$. This implies $||d_k|| > \varepsilon $. However, this contradicts the hypothesis that $||d_k|| \leq \varepsilon $ in the lemma. \Halmos  
\end{proof}

\begin{theorem} \label{optimality}
    If $k$ is the smallest index for which $||d_k|| < \varepsilon$, then we will have 
    \begin{equation*}
        |f(x^*) - f(\bar{x}_k)| = O(\varepsilon).
    \end{equation*}
\end{theorem}

\begin{proof}{Proof}
    Consider the difference between the primal objective value 

    \begin{equation*}
        f_k(x_k) =\frac{1}{|S_k|} \sum_{s=1}^{|S_k|} h(x_k^s, \omega^s)
    \end{equation*}

    \noindent and the dual objective value

    \begin{equation*}
        L_k(x_k,\lambda_k) = \frac{1}{|S_k|} \sum_{s=1}^{|S_k|} [h(x_k^s, \omega^s) + \langle \lambda_k^s, x_k^s - \bar{x}_k \rangle + \frac{\rho}{2} ||x_k^s - \bar{x}_k||^2].
    \end{equation*}

    \noindent Since $||d_k|| < \varepsilon$, from Lemma \ref{d_k^*}, we have $||g_k|| = ||x_k - \bar{x}_k|| \leq 4 \varepsilon$. Thus, we have 

    \begin{equation} \label{duality gap}
        L_k(x_k,\lambda_k) - f_k(x_k) = \frac{1}{|S_k|} \sum_{s=1}^{|S_k|} [\langle \lambda_k^s, x_k^s - \bar{x}_k \rangle + \frac{\rho}{2} ||x_k^s - \bar{x}_k||^2] = O(\varepsilon).
    \end{equation}

    \noindent Let $x_k^* = \argmin f_k(x_k)$. From the duality theorem, we have 

    \begin{equation} \label{duality theorem}
        L_k(x_k,\lambda_k) \leq f_k(x_k^*) \leq f_k(x_k).
    \end{equation}

    \noindent  \eqref{duality gap} and \eqref{duality theorem} imply that 

    \begin{equation} \label{difference 1}
        |f_k(x_k^*) - f_k(x_k)| = O(\varepsilon).
    \end{equation}

    \noindent On the other hand, using Assumption \ref{Boundedness of subgradient}, we have for every scenario $s$,

    \begin{equation*}
        |h(x_k^s,\omega^s) - h(\bar{x}_k, \omega^s)| \leq G \cdot |x_k^s - \bar{x}_k|.
    \end{equation*}

    \noindent Thus, we have 

    \begin{equation} \label{difference 2}
        |f_k(x_k) - f_k(\bar{x}_k)| = O(\varepsilon).
    \end{equation}

    \noindent Eq. \eqref{difference 1} and Eq. \eqref{difference 2} imply that 

    \begin{equation} \label{difference 3}
        |f_k(x_k^*) - f_k(\bar{x}_k)| = O(\varepsilon).
    \end{equation}

    \noindent Also, let $x^* = \argmin f(x)$, we will show that 

    \begin{equation} \label{difference 4}
        |f(x^*) - f(x_k^*)| = O(\varepsilon).
    \end{equation}

    \noindent First, note that we can use the same concentration inequality in Theorem \ref{concentration inequality} to show that with probability at least $(1-\alpha)^2$,  

    \begin{equation} \label{difference 5}
        |f_k(x^*) - f(x^*)| \leq \varepsilon \quad \textit{and} \quad |f_k(x_k^*) - f(x_k^*)| \leq \varepsilon. 
    \end{equation}

    \noindent Thus, 

    \begin{equation*}
        f_k(x^*) \leq f(x^*) + \varepsilon \quad \textit{and} \quad f(x_k^*) \leq f_k(x_k^*) +\varepsilon. 
    \end{equation*}

    \noindent Based on the definition of $x^*$ and $x_k^*$, we have

    \begin{equation*}
        f(x_k^*) \leq f_k(x_k^*) +\varepsilon \leq f_k(x^*) + \varepsilon \leq f(x^*) + 2\varepsilon,
    \end{equation*}

    \noindent which proves Eq. \eqref{difference 4}.  Finally, we conclude the proof by combining Eq. \eqref{difference 3} - Eq. \eqref{difference 5}. \Halmos   
\end{proof}

\section{Preliminary Computational Results} 

Our computational experiments are based on data sets available at the USC 3D Lab~\footnote{https://github.com/USC3DLAB/SD}. Specifically, these problems are two-stage stochastic linear programming problems, and we have tested LandS3, pgp2, 20-term, and baa99-20. The study considers three different methods: Classic PHA~\cite{RW1991}, randomized PHA~\cite{B2020}, and the sampling-based PHA presented in this paper. We track the increase in dual objective values with respect to the number of QPs solved for all algorithms. The figures also include the 95\% upper and lower bounds of the objective value estimates obtained from stochastic decomposition (SD)~\cite{sen2016} as a benchmark. Because the stopping rule for SD is based on several replications, together with the use of a compromise decision \cite{S2016}, the upper and lower bounds used for verification provide reasonably good bounds on the optimal value. The algorithms of this paper were implemented on a MacBook Pro 2023 with a 2.6GHz 12-core Apple M3 Pro processor and 32GB of 2400MHz DDR4 onboard memory. The code used in this research is written in Julia and is available at the following GitHub repository: \href{https://github.com/yhz0/progressive_hedging_scs}{Progressive Hedging SCS} (accessed on August 24, 2024), and the computational results are shown in Figure \ref{scs pha}.

\begin{figure}[H]
     \centering
     \caption{Dual objective values for different combinations (data,algorithm).}
     \begin{subfigure}[b]{0.45\textwidth}
         \centering
        \includegraphics[width=\textwidth]{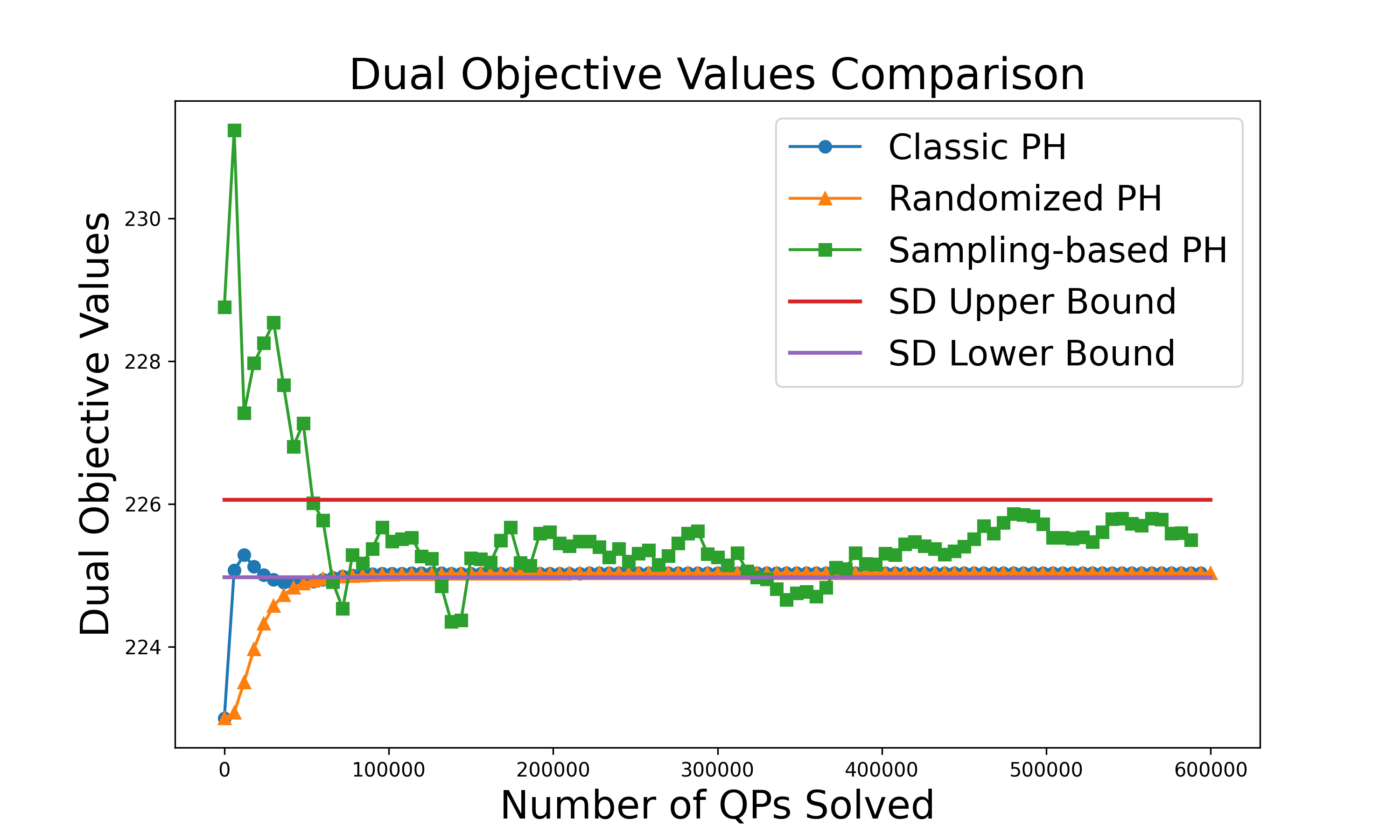}
         \caption{LandS3}
     \end{subfigure}
     \hfill
     \begin{subfigure}[b]{0.45\textwidth}
         \centering
         \includegraphics[width=\textwidth]{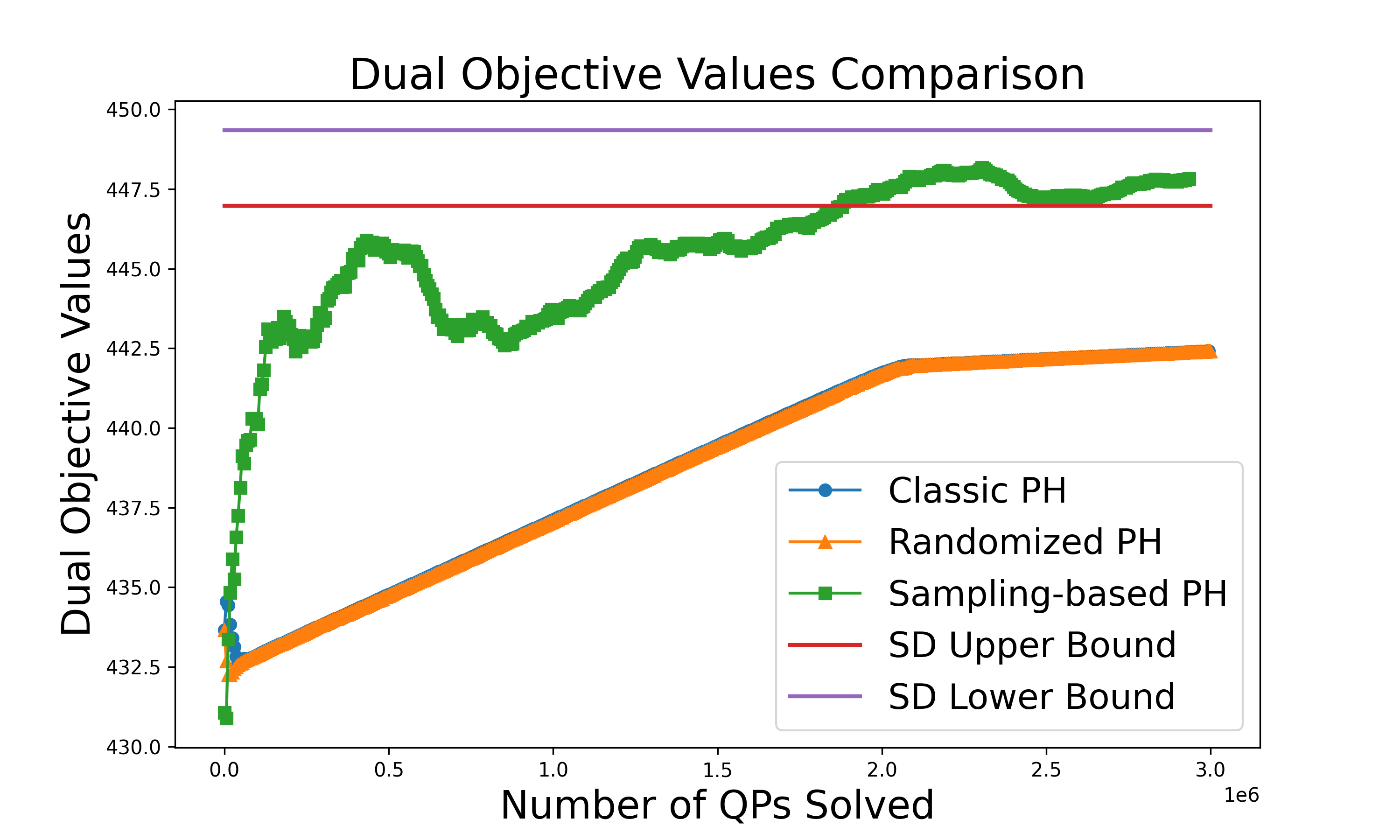}
         \caption{pgp2}
     \end{subfigure} \\
     \begin{subfigure}[b]{0.45\textwidth}
         \centering
         \includegraphics[width=\textwidth]{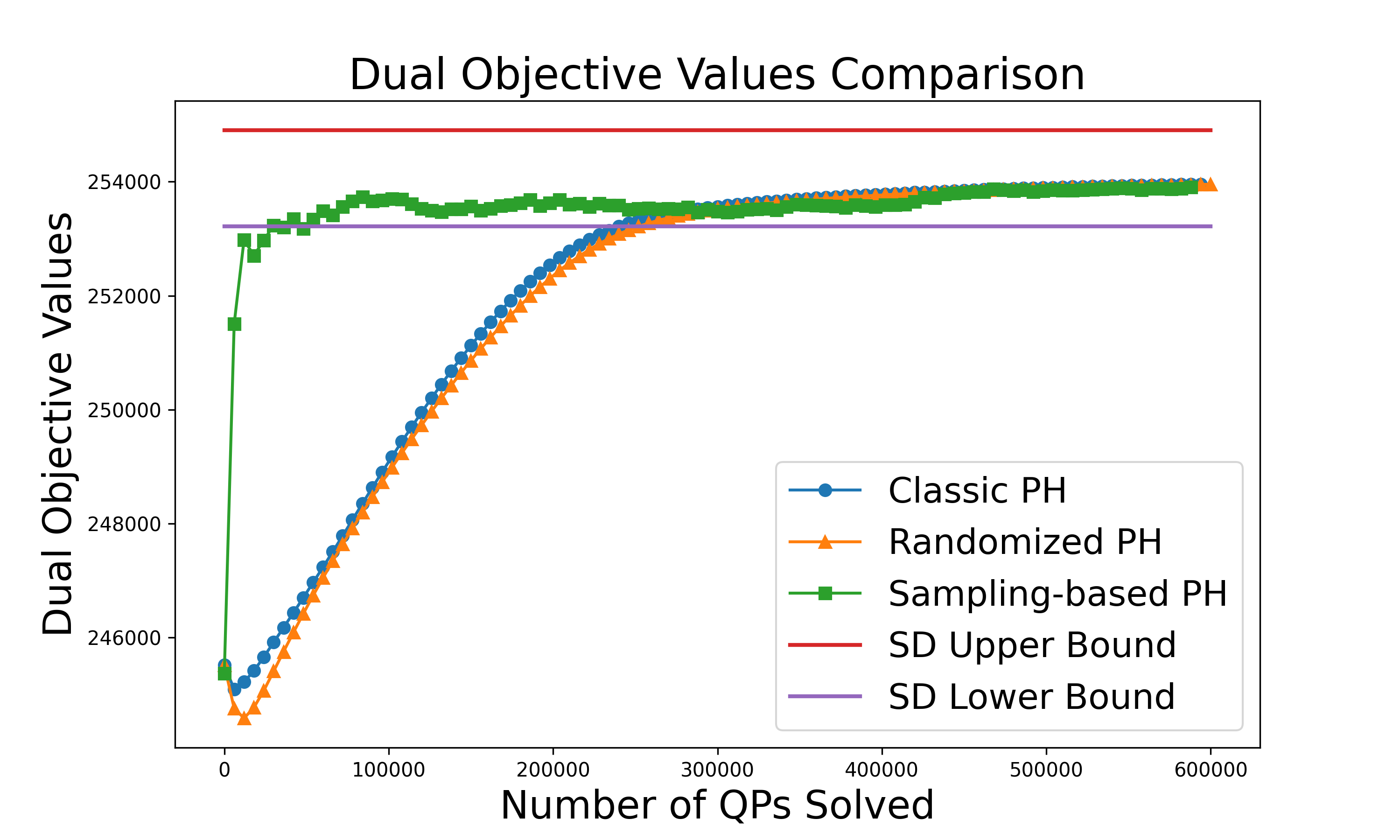}
         \caption{20-term}
     \end{subfigure}
     \hfill
     \begin{subfigure}[b]{0.45\textwidth}
         \centering
         \includegraphics[width=\textwidth]{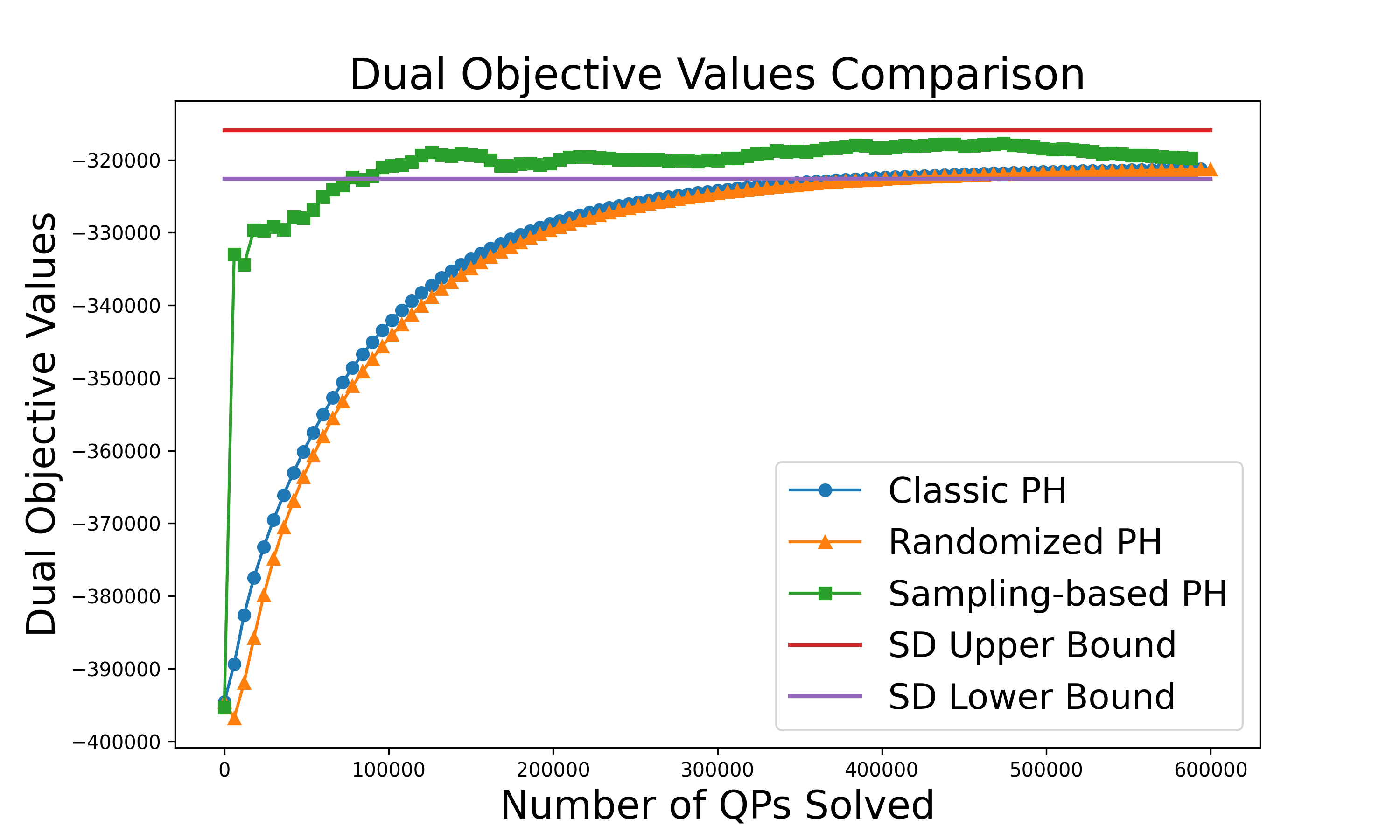}
         \caption{baa99-20}
     \end{subfigure}
        \label{scs pha}
\end{figure}

\noindent \textbf{Remarks}: (a) The objective function values of sampling-based PHA reflect the submartingale property demonstrated in Corollary \ref{submartingale}. (b) For problems which require more scenarios to converge, the sampling-based PH algorithm tends to converge faster (smaller number of scenarios) than classic PH algorithm as well as the randomized PH algorithm. In other words, the adaptive sampling-based PH algorithm is more data-efficient and more scalable than other PH algorithms for the instances explored in this study. 

\section{Conclusion}

We propose a novel adaptive sampling-based PH algorithm designed for large-scale stochastic optimization problems. Compared to the classic PH algorithm, the proposed approach introduces several key innovations:

\begin{enumerate} \item It does not require explicit knowledge of the distribution of random variables, greatly broadening the applicability of the PH algorithm. \item By leveraging sample complexity analysis and concentration inequalities, it adaptively updates the sample size, making the algorithm more robust and stochastic compared to the traditional PH algorithm, which relies solely on SAA. \item It employs stochastic conjugate subgradient directions and line-search techniques to update the dual variables, ensuring the submartingale property. This property guarantees a sufficient increase in expectation and serves as a critical foundation for establishing the convergence rate. \item Unlike many existing methods that update the penalty parameter through computational experimentation~\cite{Z2016}, our approach explicitly updates it based on optimization principles, further enhancing the algorithm's efficiency. \end{enumerate}

The convergence and convergence rate of the proposed algorithm are established by combining classical optimization techniques with modern stochastic frameworks. Specifically, Theorem \ref{contraction dec} demonstrates that the sequence of dual decision variables satisfies a contraction property in expectation, analogous to the contraction mapping principle in deterministic optimization. This property is a foundational concept also emphasized in the original PH paper~\cite{RW1991}. Furthermore, Theorem \ref{Convergence rate} leverages the submartingale property and the Renewal-Reward Theorem—a modern stochastic framework~\cite{B2014}—to establish the convergence rate of the algorithm.

In addition to the theoretical guarantees, our computational results highlight the algorithm's strong performance across several test instances. In these cases, the algorithm achieved objective function values within an acceptable range, as benchmarked against a well-tested version of the SD algorithm. The SD algorithm incorporates numerous features to ensure reliable outputs (e.g., compromise decisions~\cite{S2016}), thereby indicating that the solutions generated by our proposed method are equally reliable.

%\THEEndNotes

% Appendix here
% Options are (1) APPENDIX (with or without general title) or
%             (2) APPENDICES (if it has more than one unrelated sections)
% Outcomment the appropriate case if necessary
%
% \begin{APPENDIX}{<Title of the Appendix>}
% \end{APPENDIX}
%
%   or
%
% \begin{APPENDICES}
% \section{<Title of Section A>}
% \section{<Title of Section B>}
% etc
% \end{APPENDICES}

% Acknowledgments here
\ACKNOWLEDGMENT{The research reported in this paper was funded by DE‐SC0023361, subaward G002122-7510 from the Department of Energy under the Advanced Scientific Computing Research program.}

% References here (outcomment the appropriate case)

% CASE 1: BiBTeX used to constantly update the references
%   (while the paper is being written).
%\bibliographystyle{informs2014} % outcomment this and next line in Case 1
%\bibliography{<your bib file(s)>} % if more than one, comma separated

%\bibliographystyle{informs2014} % outcomment this and next line in Case 1
%\bibliography{sample} % if more than one, comma separated

% CASE 2: BiBTeX used to generate mypaper.bbl (to be further fine tuned)
%\input{mypaper.bbl} % outcomment this line in Case 2

%If you don't use BiBTex, you can manually itemize references as shown below.

%\bibliographystyle{nonumber}

\bibliographystyle{abbrv}
\bibliography{main}

%%%%%%%%%%%%%%%%%
\end{document}